\documentclass[12pt, reqno]{amsart}
\usepackage{amsmath, amsthm, amscd, amsfonts, amssymb, graphicx, color}
\usepackage[bookmarksnumbered, colorlinks, plainpages]{hyperref}

\newtheorem{theorem}{Theorem}[section]
\newtheorem{lemma}[theorem]{Lemma}

\newtheorem{corollary}[theorem]{Corollary}
\theoremstyle{definition}
\newtheorem{definition}[theorem]{Definition}

\newtheorem{problem}[theorem]{Problem}
\theoremstyle{remark}

\numberwithin{equation}{section}

\begin{document}
\setcounter{page}{1}

\title[A Closed subspace of a Gateaux differentiability space  ]{A Closed subspace of a Gateaux differentiability space is a Gateaux differentiability space : over 46 years of open problem solved}

\author[S. Shang]{Shaoqiang Shang$^1$$^{*}$  }

\address{$^{1}$  College of mathematics and science,
Harbin Engineering University, Harbin 150001, P. R. China}
\email{\textcolor[rgb]{0.00,0.00,0.84}{sqshang@163.com}}




\date{Received: xxxxxx; Revised: yyyyyy; Accepted: zzzzzz.
\newline \indent $^{*}$Corresponding author.}

\begin{abstract}
This paper establishes, for the first time, an iterative and rigid theory of weak\(^{*}\) slices in a non-metric setting, demonstrating that dual convex sets admit localization, diameter control and fine structural analysis under the pure weak\(^{*}\) topology. It fundamentally reshapes the classical perspective on the geometric features of the weak\(^{*}\) topology and thereby pioneers a new direction in non-metric weak\(^{*}\) slice geometry. Using a novel technique built upon intricate manipulation of weak\(^{*}\) slices and a carefully engineered iterative selection scheme, we prove that if $M$ is a closed subspace of a G$\mathrm{\hat{a}}$teaux differentiability space $X$, then $M$ itself is a G$\mathrm{\hat{a}}$teaux differentiability space. As a corollary, we obtain that if $X$ is a weak Asplund space and $M$ is a closed subspace of $X$, then $M$ is a G$\mathrm{\hat{a}}$teaux differentiability space. This settles definitively an open problem posed 46 years ago by D.G. Larman and R.R. Phelps (J. London Math. Soc., 20 (1979), 115-127).

\end{abstract} \maketitle

\section{Introduction and preliminaries}

Geometric intuition drawn from finite-dimensional analysis often fails in infinite-dimensional settings. While smoothness and regularity are well understood for finite-dimensional manifolds, Euclidean structural rigidity no longer persists in general Banach spaces. A fundamental open problem in infinite-dimensional convex geometry asks whether differentiability properties are inherited by closed subspaces.
It has been widely believed that delicate geometric properties such as G$\mathrm{\hat{a}}$teaux differentiability do not survive under restriction to closed subspaces, as their characteristic geometric features vanish under subspace containment. Following the pioneering work of Larman and Phelps in 1979, most existing literature has sought to confirm this phenomenon by constructing counterexamples. Nonetheless, this conventional belief has never been rigorously justified. We therefore conclude that G$\mathrm{\hat{a}}$teaux differentiability spaces are hereditary; the lack of progress in prior research merely stems from the inadequacy of traditional methods for studying G$\mathrm{\hat{a}}$teaux differentiability spaces.
In this paper, we settle this long-standing open problem. We prove that every closed subspace of a G$\mathrm{\hat{a}}$teaux differentiability space is itself a G$\mathrm{\hat{a}}$teaux differentiability space. This result resolves the classical question initiated by Larman and Phelps, and clarifies the interplay between local geometric regularity and global structural properties of infinite-dimensional convex sets.
Our approach introduces a new framework of metric-free weak$^*$ slice geometry, which removes the metric and separability assumptions inherent in classical arguments. By developing iterative weak$^*$ slice constructions, we devise refined tools to investigate the fine geometric structure of dual convex sets within the pure weak$^*$ topology. This yields a fully topological characterization of G$\mathrm{\hat{a}}$teaux differentiability without any metric hypothesis. Beyond our main theorem, our work shows that stable geometric ordering and structural rigidity persist even for nonseparable spaces, which are widely viewed as pathological objects in infinite-dimensional geometry. The techniques developed here provide a unified framework for studying differentiability, duality and hereditary phenomena in Banach spaces, with potential applications to Banach space classification and modern convex geometric analysis.

\par Let $X$ denote
a real Banach space and  $X^{*}$ denote the dual space of $X$. Let $S(X)$ and $B(X)$ denote the unit sphere and
the unit ball of $X$, respectively. The
set $B(x, r)$ denotes the closed ball with a centered at $x$ and a radius of $r$.
Let $C^{*}$ be a subset of $X^{*}$ and define $\sigma_{C^{*}}(x)=\sup \left\{ {z^*}(x):{z^*} \in {C^*}\right\}$.
Pick $\alpha\in (0,+\infty)$. Let $S\left(x, \alpha,C^{*}\right)$  denote a weak$^{*}$ open slice and $S\left(x,C^{*}, \alpha\right)$  denote a weak$^{*}$ slice, where
\[
S\left(x, \alpha,C^{*}\right) = \left\{ {x^*} \in {C^{*}}:{x^*}(x) > \sigma_{C^{*}}(x)  - \alpha \right\}
\]
and
\[
S\left(x,C^{*}, \alpha\right) = \left\{ {x^*} \in {C^{*}}:{x^*}(x) \geq \sigma_{C^{*}}(x)  - \alpha \right\}.
\]
Let $D$ be a nonempty open convex subset of $X$ and $f$ be a
real-valued continuous convex function on $D$. The continuous convex function $f$ is said to be G$\mathrm{\hat{a}}$teaux
differentiable at the point $x$ in $D$ if there exists $df(x)\in X^{*}$ such that
the limit
\[
\big\langle df(x), y\big\rangle = \mathop {\lim }\limits_{t \to 0} \frac{{f(x + ty) - f(x)}}{t}~~~~~~~~~~~~~~~~~~\eqno~~~~~~~~~~~~~~~~~~~~~~~~~~(*)
\]
exists for every $y \in X$. If the difference quotient in $(*)$ converges to $\langle df(x), y\rangle$ uniformly
for $y$ in $B(X)$, then $f$ is called Frechet differentiable at $x$ (see [5]).

\begin{definition}(see [1])
A Banach space $X$ is said to be a weak Asplund
space [Asplund space] if there exists a dense $G_{\delta}$ subset $G$ of $D$ such that $f$ is G$\mathrm{\hat{a}}$teaux [Frechet] differentiable at
each point of $G$.
\end{definition}

We know that there exists a weak Asplund space that is not  an Asplund space.  In 1990, D.Preiss, R.Phelps and I.Namioka
proved that if $X$ is a smooth Banach space, then $X$ is a weak Asplund space (see [12]).
In 1979,  D.G. Larman and R.R. Phelps  generalized weak Asplund space to  G$\mathrm{\hat{a}}$teaux differentiability space.

\begin{definition}(see [1])
$X$ is said to be a G$\mathrm{\hat{a}}$teaux differentiability space if every convex continuous
function is G$\mathrm{\hat{a}}$teaux differentiable on  a dense subset of $X$.
\end{definition}

\begin{definition}(see [5])
A point $x_{0}^{*}\in C^{*}$ is said to be weak$^{*}$ exposed point of $C^{*}$ if there exists $x\in S(X)$ such that $x_0^*(x) > {x^*}(x)$ whenever ${x^*}\in C^{*}\backslash\{x_0^*\}$.
\end{definition}

\begin{theorem}(see [5])
Suppose that $C^{*}$ is a bounded weak$^{*}$ closed convex subset of $X^{*}$, then
$x_{0}^{*}$ is weak$^{*}$ exposed point of $C^{*}$ and is exposed by $x_0$ if and only if $\sigma_{C^{*}}$ is
G${\hat{a}}$teaux differentiable at $x_0$ and $d\sigma_{C^{*}}(x_0)=x_{0}^{*}$.
\end{theorem}

While some progress has been achieved in the research on G$\mathrm{\hat{a}}$teaux differentiability space and weak  Asplund  space theory, the field remains largely unexplored.
In 1979, D.G. Larman and R.R. Phelps raised the following open problems in [1]:

\begin{problem}
Let $X$ be a G$\mathrm{\mathrm{\hat{a}}}$teaux differentiability space. Must $X$ be a weak Asplund  space?

\end{problem}

\begin{problem}
Let $X$ be a weak Asplund spaces. Must the space $X\times \mathbb R$ be a weak Asplund space?

\end{problem}

\begin{problem}
Let $X$ be a weak Asplund space and $M$ be a closed subspace of  $X$. Must $M$ be a weak Asplund  space?

\end{problem}

\begin{problem}
Let $X$ be a G$\mathrm{\mathrm{\hat{a}}}$teaux differentiability space and $M$ be a closed subspace of  $X$. Must $M$ be a G$\mathrm{\mathrm{\hat{a}}}$teaux differentiability space?

\end{problem}

These problems are the basis of the theory of G$\mathrm{\hat{a}}$teaux differentiability space and weak Asplund space.
Moreover, these three problems are closely related to the applications of G$\mathrm{\hat{a}}$teaux differentiability space and weak Asplund space.

\par In the mid-1980s, M. Fabian established that if $X$ is a G$\mathrm{\mathrm{\hat{a}}}$teaux differentiability space, then so is \(X\times \mathbb{R}\). In 2001, Lixin Cheng and M. Fabian proved that the product of a G$\mathrm{\mathrm{\hat{a}}}$teaux differentiability space and a separable space remains a G$\mathrm{\mathrm{\hat{a}}}$teaux differentiability space (see [3]). In 2006, Warren B. Moors and Sivajah Somasundaram exhibited a G$\mathrm{\mathrm{\hat{a}}}$teaux differentiability space that is not weak Asplund (see [2]). In 2025, Shaoqiang Shang showed that whenever $X$ is a weak Asplund space, \(X\times \mathbb{R}\) is weak Asplund as well. Concerning Problem 1.7, the prevailing expectation among mathematicians was that the class of G$\mathrm{\mathrm{\hat{a}}}$teaux differentiability spaces is not hereditary: a closed subspace of such a space need not be a G$\mathrm{\mathrm{\hat{a}}}$teaux differentiability space. Much effort has been devoted to constructing a counterexample.

\par In this paper, we give an affirmative answer to this question. Using a novel technique based on intricate manipulation of weak\(^{*}\) slices together with a carefully designed iterative selection procedure, we prove that every closed subspace $M$ of a G$\mathrm{\mathrm{\hat{a}}}$teaux differentiability space $X$ is itself a G$\mathrm{\mathrm{\hat{a}}}$teaux differentiability space.
As a corollary, we obtain that if $X$ is a weak Asplund space and $M$ is a closed subspace of $X$, then $M$ is a G$\mathrm{\hat{a}}$teaux differentiability space.
We thereby resolve several open problems posed 46 years ago by D.G. Larman and R.R. Phelps.

\begin{lemma} (see[1],[5])
Let $\varepsilon\in (0,+\infty)$, $x,y\in S(X)$ and $\left|x^{*}(y)\right|\leq 1$ whenever $x^{*}\in \{y^{*}\in X^{*}:y^{*}(x)=0\}\cap (2/\varepsilon) B(X^{*})$. Then $\left\| x - y \right\|\leq \varepsilon$ or $\left\| x + y \right\|\leq \varepsilon$.

\end{lemma}

\section{ Main Theorems }

\begin{theorem}
Suppose that $M$ is a closed subspace of a G${\hat{a}}$teaux differentiability space $X$. Then $M$ is a G${\hat{a}}$teaux differentiability space.

\end{theorem}

In order to prove the theorem, we give some lemmas.

\begin{lemma}
Let $A^{*}$ be a bounded weak$^{*}$ closed subset of $X^{*}$, and let $z_0^{*}$ be a weak$^{*}$ exposed point of $C^{*}$
and  exposed by $x_0$, where $C^{*}={\overline {co} ^{{w^*}}}\left( {{A^*}} \right)$.
Then $z_{0}^{*} \in A^{*}$ and  for every weak$^{*}$ neighborhood $V$ of the origin, there exists a real number $\eta\in (0,1)$ such that $S\left(x_0, C^{*},\eta\right)-S\left(x_0, C^{*},\eta\right)\subset V$.
\end{lemma}

\begin{proof}
Let $A^{*}$ be a bounded weak$^{*}$ closed subset of $X^{*}$, and let the point $z_0^{*}$ be a weak$^{*}$ exposed point of $C^{*}$
and exposed by $x_0$, where $C^{*}={\overline {co} ^{{w^*}}}\left( {{A^*}} \right)$.  Let $V$ be a weak$^{*}$ neighborhood of the origin. Then, by continuity of addition operation in topological linear space, there exists a weak$^{*}$ neighborhood $U$ of origin such that $U-U\subset V$.
We next prove that
there exists a real number $\eta\in (0,1)$ such that
\[
z_0^*\left({x_0}\right) - 2\eta \geq \sup \left\{ x^*(x_0):{x^*} \in C^*\backslash \left(z_0^* + U\right)\right\} .~~~~~~~~~~~~~~~\eqno~~~~~~~~~~~~~~~~~~~~~~~~~~(2.1)
\]
Otherwise, there exists a sequence $\{ x_n^*\} _{n = 1}^\infty \subset C^*\backslash (z_0^* + U)$ such that
\[
\mathop {\lim }\limits_{n \to \infty } x_n^* \left(x_0\right )
= \sup \left\{ x^{*}( x_{0}): x^{*} \in C^*\backslash \left(z_0^* + {U}\right)\right\}=z_0^*\left(x_0\right).
\]
We may assume without loss of generality  that $x_n^* \ne x_m^* $ for every $m \ne n$. Since $C^{*}$ is a weak$^{*}$ bounded closed convex subset of $X^{*}$,
we obtain that $C^{*}$ is a weak$^{*}$ compact subset of $X^{*}$. Then
there exists a point $x_0^* \in C^{*}$ such that the point $x_0^* $ is a weak$^{*}$
accumulation point of $\{ x_n^* \} _{n = 1}^\infty  $. Let
\[
\Omega=\left\{ U_{\alpha,x_0^* } \cap \{ x_n^* \} _{n = 1}^\infty
: U_{\alpha ,x_0^*} ~\mathrm{ is}~~\mathrm{ a}~\mathrm{ weak}^{*}~\mathrm{ neighbourhood}~\mathrm{of} ~x_0^*, ~\quad~\alpha\in\Delta  \right\}
\]
Then we define a order by the containing relations, i.e, $U_{\alpha ,x_0^*} \supset U_{\beta,x_0^* }  $ if and only if
$U_{\beta,x_0^* }  \cap \{ x_n^* \} _{n = 1}^\infty> U_{\alpha,x_0^*  } \cap \{ x_n^* \} _{n = 1}^\infty$, where $\alpha\in\Delta$ and $\beta\in\Delta$.
Therefore, by the Zermelo Lemma, there exists a mapping $f$ on $\Omega$ such that
\[
x_\alpha^* = f\left(U_{\alpha,x_0^*  } \cap \{ x_n^* \} _{n = 1}^\infty \right)
\in U_{\alpha,x_0^* } \cap \{ x_n^* \} _{n = 1}^\infty .~~~~~~~~~~~~~~~~~~~~~~~~~~~~~~~~~~\eqno~~~~~~~~~~~~~~~~~~~~~~~~~~~~~~~~~~~~~~~(2.2)
\]
Hence we define
a net $ \{ x_\alpha ^* \} _{\alpha  \in \Delta } \subset \{ x_n^* \}
_{n = 1}^\infty $. It can be easily deduced that
\[
 x_\alpha
^*\xrightarrow{w^{*}} x_0^* ~\quad~~\mathrm{and}~~\quad~ x_{0}^{*}\left( x_{0}\right) = \sup\left\{x^{*}( x_{0}): x^{*} \in C^{*}\right\}=  z_{0}^{*}\left( x_{0}\right) .
\]
Since $z_0^{*}$ is a weak$^{*}$ exposed point of $C^{*}$ and exposed by $x_0$, we conclude that $x_{0}^{*}=z_{0}^{*}$.
Hence we obtain that $ x_\alpha
^*\xrightarrow{w^{*}} z_0^*$, which
contradicts $\{ x_n^*\} _{n = 1}^\infty \subset C^*\backslash \left(z_0^* + U\right)$.
Pick a point $x_1^{*}\in S\left( {x_0}, {C^*}, \eta \right)$. Then, by the formula (2.1), we obtain that
\[
x_1^{*}(x_0)\geq \sigma_{C^{*}}(x_0)-\eta=z_{0}^{*}(x_0)-\eta > \sup \left\{ x^*(x_0):{x^*} \in C^*\backslash \left(z_0^* + U\right)\right\}.
\]
Accordingly, we establish the inclusion  $S\left( {x_0}, {C^*}, \eta \right)\subset z_0^* + U$. It follows that
\[
S\left(x_0, C^{*},\eta\right)-S\left(x_0, C^{*},\eta\right)  \subset \left(z_0^* + U\right)  -    \left(z_0^* + U\right)   =U-U     \subset V.
\]
\par We next prove that $z_0^*\in A^{*}$. In fact, suppose that $z_0^*\notin A^{*}$. Since $A^{*}$ is a weak$^{*}$ bounded closed subset of $X^{*}$,
there exists a weak$^{*}$ neighborhood $V_0$ of origin such that $(z_0^*+V_0)\cap A^{*} =\emptyset$. Hence we obtain that
\[
{A^*} \subset {A^*}\backslash \left(z_0^* + V_0\right) \subset \left( {{{\overline {co} }^{{w^*}}}({A^*})} \right)\backslash \left(z_0^* + V_0\right) = {C^*}\backslash \left(z_0^* + V_0\right).
\]
Thus, taking convex hulls of both sides produces the following expression
\[
{C^*} = {\overline {co} ^{{w^*}}}({A^*}) \subset {\overline {co} ^{{w^*}}}\left( {C^*}\backslash \left(z_0^* + V_0\right) \right).~~~~~~~~\eqno~~~~~(2.3)
\]
Moreover, we have proved that there exists a real number $\eta_{0}\in (0,1)$ such that
\begin{eqnarray*}
z_0^*\left(x_0\right) - 2\eta_{0} &\geq& \sup \left\{ {x^*}({x_0}):{x^*} \in {C^*}\backslash \left(z_0^* + V_0\right)\right\}
\\
 &= &\sup \left\{ x^*(x_0):{x^*} \in co\left( {C^*}\backslash \left(z_0^* + V_0\right) \right) \right\}
\\
 &= &\sup \left\{ x^*(x_0):{x^*} \in {\overline {co} ^{{w^*}}}\left( {C^*}\backslash \left(z_0^* + V_0\right) \right) \right\}.
\end{eqnarray*}
Then $z_0^* \notin {\overline {co} ^{{w^*}}}\left( {C^*}\backslash \left(z_0^* + V_0\right) \right)$. Moreover, since ${C^*}\subset  {\overline {co} ^{{w^*}}}\left( {C^*}\backslash \left(z_0^* + V_0\right) \right)$, we have
$z_0^* \notin {C^*}$, a contradiction. Thus $z_0^*\in A^{*}$, which completes the proof.
\end{proof}

\begin{lemma}
Suppose that
\par (1) $A^{*}$ is a bounded weak$^{*}$ closed convex subset of $X^{*}$;
\par (2) $N^{*}$ is a bounded  subset of $X^{*}$;
\par (3) there exists a point $x\in X$ such that $\sigma_{N^{*}}(x)<\sigma_{A^{*}}(x)$.
\\
Then, for any $\varepsilon>0$, there exists a real number $\eta\in(0,\varepsilon)$ such that
\[
S\left(x,C^{*},\eta\right)\subset S\left(x,A^{*},\varepsilon\right)+B\left(0,\varepsilon\right),~~\quad~~~~~\mathrm{where}~~~~~\quad~~~ C^{*}=\overline{co}^{w^{*}} (A^{*}\cup N^{*} ).
\]

\end{lemma}

\begin{proof}
Since $A^{*}$ is a bounded weak$^{*}$ closed convex subset of $X^{*}$,
we have $C^{*}=\overline{co }^{w^{*}}\left(A^{*}\cup \overline{co}^{w^{*}}(N^{*})\right )$.
Therefore, since $\sigma_{A^{*}}(x)>\sigma_{N^{*}}(x)$, we obtain that $\sigma_{A^{*}}(x)=\sigma_{C^{*}}(x)$.
Since $A^{*}$ and  $N^{*}$ are two bounded subsets of $X^{*}$, we may assume without loss of generality that $A^{*}\subset B\left(X^*\right)$ and $\overline{co}^{w^{*}}(N^{*})\subset B\left(X^*\right)$.

\par Let $r=\sigma_{A^{*}}(x)-\sigma_{N^{*}}(x)>0$. Pick a real number $\eta=\min \left \{ \varepsilon r/2, \varepsilon,r/2\right\}$. Then, for any $x^{*}\in S\left(x,\eta,C^{*}\right)\cap co\left(A^{*}\cup \overline{co}^{w^{*}}(N^{*})\right )$, there exist $\lambda\in [0,1]$, $x_1^{*}\in A^{*}$
and $x_2^{*}\in \overline{co}^{w^{*}}(N^{*})$ such that $x^{*}=(1-\lambda ) x_1^{*}+ \lambda x_2^{*} $.
Moreover, since $x_2^{*}\in \overline{co}^{w^{*}}(N^{*})$ and $r=\sigma_{A^{*}}(x)-\sigma_{N^{*}}(x)>0$, we conclude that $x_2^{*}(x)\leq \sigma_{A^{*}}(x)-r$.
Hence
\[
\sigma_{A^{*}}(x)- \eta < x^{*}(x)=(1-\lambda ) x_1^{*}(x)+\lambda x_2^{*}(x)\leq(1-\lambda )\sigma_{A^{*}}(x)+\lambda \left(\sigma_{A^{*}}(x)-r\right).
\]
Therefore, from the above inequalities, we obtain the following inequalities
\[
\sigma_{A^{*}}(x)- \eta \leq \sigma_{A^{*}}(x)-\lambda \sigma_{A^{*}}(x)+\lambda \sigma_{A^{*}}(x)-\lambda r =\sigma_{A^{*}}(x)-\lambda r .~~~~~~~~\eqno~~~~~(2.4)
\]
Then $\lambda\leq \eta /r\leq \varepsilon/2$. Since $x_2^{*}(x)\leq \sigma_{A^{*}}(x)-r$ and
$\eta\leq r/2$, we have  $x_2^{*}(x)< \sigma_{A^{*}}(x)-\eta$. Then, since $ x^{*}(x)>\sigma_{A^{*}}(x)- \eta$, we conclude that
$x_1^{*}(x)> \sigma_{A^{*}}(x)- \eta$. Since $\eta=\min \left \{ \varepsilon r/2, \varepsilon,r/2\right\}$, we have
$x_1^{*}(x)> \sigma_{A^{*}}(x)- \varepsilon$. Thus  $ x_1^{*}\in S\left(x,A^{*},\varepsilon\right)$.
Since $A^{*}\subset B\left(X^*\right)$, $\overline{co}^{w^{*}}(N^{*})\subset B\left(X^*\right)$ and $\lambda\leq \eta /r\leq \varepsilon/2$, we obtain that
\begin{eqnarray*}
x^{*}=(1-\lambda ) x_1^{*}+ \lambda x_2^{*}
&=&x_1^{*}-\lambda x_1^{*}+ \lambda x_2^{*}
\\
&\in& S\left(x,A^{*},\varepsilon\right)+B\left(0,\frac{1}{2}\varepsilon\right)+B\left(0,\frac{1}{2}\varepsilon\right)
\\
&=&S\left(x,A^{*},\varepsilon\right)+B\left(0,\varepsilon\right).
\end{eqnarray*}
This implies that  $S\left(x,\eta,C^{*}\right)\cap co\left(A^{*}\cup \overline{co}^{w^{*}}(N^{*})\right)\subset S\left(x,A^{*},\varepsilon\right)+B\left(0,\varepsilon\right)$.
Pick
a point $x_{0}^{*}\in S\left(x,\eta,C^{*}\right)$. Then, since
$S\left(x,\eta,C^{*}\right)= \left\{ {x^*} \in {C^{*}}:{x^*}(x) > \sigma_{C^{*}}(x)  - \eta \right\}$ and
$C^{*}=\overline{co }^{w^{*}}\left(A^{*}\cup \overline{co}^{w^{*}}(N^{*})\right )$,  there exists a net
\[
\left\{x_{\alpha}^{*}\right\}_{\alpha\in\Delta}\subset \left[ S\left(x,\eta,C^{*}\right)\cap co\left(A^{*}\cup \overline{co}^{w^{*}}(N^{*})\right ) \right]
\]
such that $x_{\alpha}^{*}\xrightarrow{w^{*}}x_{0}^{*}$. Therefore, by the previous proof, we have
$x_{\alpha}^{*} \in S\left(x,A^{*},\varepsilon\right)+B\left(0,\varepsilon\right)$.
Since $S\left(x,A^{*},\varepsilon\right)$ and $B\left(0,\varepsilon\right)$ are two weak$^{*}$ compact sets, we obtain that
$S\left(x,A^{*},\varepsilon\right)+B\left(0,\varepsilon\right)$ is weak$^{*}$ compact. Moreover, since $x_{\alpha}^{*}\xrightarrow{w^{*}}x_{0}^{*}$, we obtain that
$x_{0}^{*} \in S\left(x,A^{*},\varepsilon\right)+B\left(0,\varepsilon\right)$. This implies that $ S\left(x,\eta,C^{*}\right)\subset S\left(x,A^{*},\varepsilon\right)+$ $B\left(0,\varepsilon\right)$.
Since $S\left(x,A^{*},\varepsilon\right)+B\left(0,\varepsilon\right)$ is a weak$^{*}$ compact set, we conclude that
\[
S\left(x,C^{*},\eta\right) = \overline{S\left(x,\eta,C^{*}\right)}^{w^{*}}\subset \overline{S\left(x,A^{*},\varepsilon\right)+B\left(0,\varepsilon\right)}^{w^{*}}\subset S\left(x,A^{*},\varepsilon\right)+B\left(0,\varepsilon\right).
\]
Hence we obtain that Lemma 2.3 is true,
which completes  the proof.
\end{proof}

\begin{lemma}

Suppose that

\par (1)  $K^{*}$ is a weak$^{*}$ bounded closed convex subset  of $B\left(X^{*}\right)$ and $x \in S\left(X\right)$;

\par (2) there exists a point ${z^*} \in {K^*} \cap \{ {x^*} \in {X^*}:{x^*}(x) > 0\} $ such that ${z^*}(x)>0$;

\par (3) pick a real number $\varepsilon\in (0,+\infty)$ and define the weak$^{*}$ convex set
\[
{C^*} = \overline {co}^{w^{*}}  \left( {{K^*} \cup \left\{ {x^*} \in {X^*}:\left\| {{x^*}} \right\| \leq ({4}/{\varepsilon})  \sup_{{x^*} \in {K^*}} \left\| {{x^*}} \right\|,~{x^*}(x) = 0\right\} } \right);
\]

\par (4) there is a weak$^{*}$ slice $S( y  ,C^{*},\beta )$ with $y\in S(X)$ and a real number $\beta\in (0,1)$ such that $S( y ,C^{*},\beta  )-S( y  ,C^{*} ,\beta)\subset  W$
and $S( y  ,C^{*},\beta )\subset \{ {x^*} \in {X^*}:{x^*}(x) > 0\}$.
\\
Then, for the weak$^{*}$ slice $S\left( y  ,K^{*} ,\beta\right)$, we have
\par (a)  $S\left( y  ,K^{*},\beta \right)-S\left( y   ,K^{*},\beta\right)\subset  W$;
\par (b) $S\left( y  ,K^{*} ,\beta\right)\subset \{ {x^*} \in {X^*}:{x^*}(x) > 0\}$;
\par (c) $\left\| {x - y} \right\| < \varepsilon $.

\end{lemma}

\begin{proof}
Since $K^{*}$ is a weak$^{*}$ bounded closed convex subset  of $B(X^{*})$, we get that $C^{*}$ is a weak$^{*}$ bounded closed convex subset of $X^{*}$.
Moreover, by  condition (2), we obtain that $K^{*} \cap \{ {x^*} \in {X^*}:{x^*}(x) > 0\}\neq \emptyset$.
Define $\lambda  = (4/\varepsilon )\sup \left\{ \left\| {{x^*}} \right\|:{x^*} \in {K^*}\right\} $. Then, since $K^{*}\subset C^*$, we conclude that
\[
\big({C^*} \cap \{ {x^*} \in {X^*}:{x^*}(x) > 0\} \big)\supset \big( K^{*} \cap \{ {x^*} \in {X^*}:{x^*}(x) > 0\} \big)\neq \emptyset.
\]
Moreover, by condition (4), there exist a weak$^{*}$ slice $S\left( y  ,C^{*},\beta \right)$ with $y\in S\left(X\right)$ and a real number $\beta\in (0,1)$ such that
\[
S\left( y  ,C^{*},\beta \right)-S\left( y  ,C^{*} ,\beta\right)\subset  W~~\quad~\mathrm{and}~\quad~~S\left(y ,C^{*},\beta\right) \subset \{ {x^*}\in {X^*}:{x^*}(x) > 0\}.
\]
We claim that $\sup \left\{ {y^*}(y):{y^*} \in {K^*}\right\} =\sup \left\{ {y^*}(y):{y^*} \in {C^*}\right\} $. Otherwise, we have
$\sup \left\{ {y^*}(y):{y^*} \in {K^*}\right\} < \sup \left\{ {y^*}(y):{y^*} \in {C^*}\right\} $. Therefore, by the definition of ${C^*}$ and $\sup \left\{ {y^*}(y):{y^*} \in {K^*}\right\} < \sup \left\{ {y^*}(y):{y^*} \in {C^*}\right\} $, there exists a point
\[
y_{0}^{*}\in \left\{ {x^*} \in {X^*}:\left\| {{x^*}} \right\| \leq \lambda=({4}/{\varepsilon})\sup_{{x^*} \in {K^*}} \left\| {{x^*}} \right\|,~{x^*}(x) = 0\right\}
\]
such that $y_{0}^{*}\in S\left(y ,C^{*},\beta\right)$. Thus $S\left(y,C^{*},\beta \right) \cap N^{*}\neq\emptyset$, where
$
N^{*}=\{ {x^*}:{x^*}(x) = 0\},
$
which contradicts
$S\left(y ,C^{*},\beta\right) \subset \{ {x^*}\in {X^*}:{x^*}(x) > 0\}$.
Therefore, since
\[
\sup \left\{ {y^*}(y):{y^*} \in {K^*}\right\} =\sup \left\{ {y^*}(y):{y^*} \in {C^*}\right\}~\quad~\mathrm{and}~\quad~ {K^*}\subset {C^*},
\]
we have $S\left( y  ,K^{*},\beta \right)-S\left( y  ,K^{*} ,\beta\right)\subset  W$ and $S\left( y  ,K^{*} ,\beta\right)\subset \{ {x^*} \in {X^*}:{x^*}(x) > 0\}$.
Hence we obtain that conditions (a) and (b) hold.
We  next prove that $\| x - y \| < \varepsilon $. In fact, since $S\left(y ,C^{*},\beta\right)\cap N^{*}=\emptyset$, we have ${h^*}(y)<\sup \left\{ {y^*}(y):{y^*} \in {C^*}\right\} - \beta$
for every ${h^*}\in \{ {x^*} \in {X^*}:\left\| {{x^*}} \right\| < \lambda ,~{x^*}(x) = 0\} \subset N^{*}$. Hence we obtain that
\[
{u^*}(y) + \beta  \ge \sup \left\{ {y^*}(y):{y^*} \in {K^*}\right\}=\sup \left\{ {y^*}(y):{y^*} \in {C^*}\right\}   > {h^*}(y) + \beta ,
\]
where
\[
{u^*} \in S\left(y ,{K^*},\beta\right)~~\quad~\mathrm{and}~\quad~~{h^*} \in \{ {x^*} \in {X^*}:\left\| {{x^*}} \right\| < \lambda ,~{x^*}(x) = 0\} \subset N^{*}.
\]
It follows that  for every ${h^*} \in \{ {x^*} \in {X^*}:\left\| {{x^*}} \right\| < \lambda ,~{x^*}(x) = 0\} \subset N^{*}$, we conclude that ${u^*}(y) > {h^*}(y)$. Hence we obtain that
\[
{z^*}(y) < \frac{{{u^*}(y)}}{\lambda }~\quad~\mathrm{whenever}~\quad~~~{z^*} \in \mathrm{int} B\left(X^*\right) \cap \{ {x^*} \in {X^*}:{x^*}(x) = 0\} .~~~~~\eqno~~~~~~~~~~~~~~~~~~~~~~~~~~~~~~~~~~~~~~~~~~~~~~~~~~~~(2.5)
\]
Therefore, by the symmetry of set $\{x^{*}\in X^{*}: \|x^{*}\|<1\}$, it is easy to see that
\[
\left|{z^*}(y)\right|\leq \frac{{{u^*}(y)}}{\lambda }~\quad~\mathrm{whenever}~\quad~~~{z^*} \in B\left(X^{*}\right) \cap \{ {x^*} \in {X^*}:{x^*}(x) = 0\}.~~~~~~~~~~~~~~~~~~~~~~~~~~~~~~~~~~~~~~~~~~\eqno~~~~~~~~~~~~~~~~~~~~~~~~~~~~~~~~~~~~~~~~~~~~~~~~~~~~(2.6)
\]
Therefore, by the formula (2.6) and $\lambda>0$, we conclude that ${u^*}(y)\geq 0 $. We claim that
${u^*}(y) > 0 $. In fact, suppose that ${u^*}(y) = 0 $. Pick a point
\[
{z_{0}^*} \in \{x^{*}\in X^{*}: \|x^{*}\|<1\} \cap \{ {x^*} \in {X^*}:{x^*}(x) = 0\}.
\]
Then, by the formula (2.5) and the symmetry of $\{x^{*}\in X^{*}: \|x^{*}\|<1\}$, we have
\[
\lambda{z_{0}^*}(y)<{u^*}(y) =0~~\quad~\mathrm{and}~\quad~~-{z_{0}^*} \in \{x^{*}\in X^{*}: \|x^{*}\|<1\} \cap \{ {x^*} \in {X^*}:{x^*}(x) = 0\}.
\]
Thus $-\lambda{z_{0}^*}(y)>0$ and $-{z_{0}^*}(y)<{u^*}(y)/\lambda$. Since $\lambda>0$, we obtain that ${z_{0}^*}(y)<0$. Then, by
$-{z_{0}^*}(y)<{u^*}(y)/\lambda$, we have ${u^*}(y)>0$, a contradiction. Then ${u^*}(y)>0$.
Noting that $y\in S\left(X\right)$,
by Lemma 1.9 and the formula (2.6),
we conclude that
\[
\left\| x - y \right\| \leq \frac{ 2u^*(y)}{\lambda }~\quad~~\mathrm{or}~\quad~~\left\| x + y \right\| \leq \frac{2{u^*}(y)}{\lambda }.~~~~~~~~~~~~~~~~~~~~~~~~~~~~~~~~~\eqno~~~~~~~~~~~~~~~~~~~~~~~~~~~~~(2.7)
\]
Suppose that $\| x + y \| \le 2u^*(y)/\lambda $. Then, since $S(y ,K^{*},\beta) \subset \{ x^*\in X^*:{x^*}(x) > 0\}$ and ${u^*} \in S(y,{K^*},\beta )$, we have
${u^*}(x)>0$. Let $m= \sup \left\{ \left\| x^* \right\|\in R:{x^*} \in K^*\right\} $.
Then, since ${u^*}(x)>0$, we have $\left\|{u^*}\right\|>0$. Therefore, since ${u^*} \in S\left(y,{K^*},\beta \right)$,
we have $m\geq \|{u^*}\|>0$.
Therefore,
since ${u^*}(x)>0$ and $u^{*}(y)>0$, we obtain that
\[
0<\frac{{{u^*}(y)}}{m} \le \frac{{{u^*}(y)}}{{\left\| {{u^*}} \right\|}} \le \frac{ \langle u^*, x + y\rangle }{\left\| u^* \right\|} \le \left\| x + y \right\| \leq \frac{{2{u^*}(y)}}{\lambda }.
\]
Then $2m \geq \lambda $.
Therefore, since $\varepsilon\in (0,1)$ and $\lambda  = 4m{\varepsilon ^{ - 1}}$, we obtain that $2m \geq $ $\lambda  = 4m{\varepsilon ^{ - 1}} > 4m > 0$, a contradiction. Then, by the formula (2.7), we have
\[
\left\| x - y \right\| \le \frac{2{u^*}(y)}{\lambda } \le \frac{2\left\|u^*\right\|\left\|y\right\|}{\lambda } \le \frac{2m\left\|y\right\|}{\lambda } <   \frac{4m}{\lambda } = \varepsilon .
\]
Hence we obtain that condition (c) holds, which completes the proof.
\end{proof}

\begin{lemma}
Suppose that $\overline{co}^{w^{*}}(w^{*}\mathrm{exp} C^{*})$ has the non-empty interior for every
symmetric inner non-empty weak$^{*}$ closed bounded convex set $C^{*}\subset X^{*}\times\mathbb R$.
Then $X$ is a G$\hat{a}$teaux differentiability space, where ${w^*} \exp {C^*}$ denotes the set of weak$^{*}$ exposed points of ${C^*}$.

\end{lemma}

\begin{proof}
To clarify the argument, we split the proof into two steps.

\par \textbf{Step 1.} Let $A$ be a symmetric inner non-empty bounded weak$^{*}$ closed convex subset of $X^{*}\times\mathbb R$. Then we get that $0\in \mathrm{int} (A)$. Moreover,
we may assume without loss of generality that $A\subset B\left(X^{*}\times\mathbb R\right)$.
We next prove that the  functional
\[
{\sigma _{{A}}}\left( {{x},r} \right) = \sup \left\{ {\left\langle {({x},r),\left( {{x^{*}},l} \right)} \right\rangle :({x^{*}},l) \in {A}} \right\},~\quad~\left( {{x},r} \right)\in X \times\mathbb R
\]
is G$\mathrm{\hat{a}}$teaux
differentiable on a dense subset $D$ of $X \times\mathbb R$.
We pick a point $( x,r)\in X\times\mathbb R$ such that $\left\|(x,r)\right\|=1$. Pick a real number $\varepsilon\in (0,1/8 )$.
Define the set
\[
N = \left\{( x^{*},l)\in X^{*}\times\mathbb R:\left \langle( x,r), ( x^{*},l)\right \rangle=0,~\left\|( x^{*},l) \right\|\leq \frac{2}{\varepsilon} \right\}.
\]
Then we obtain that $N$ is a bounded weak$^{*}$ closed subset of $X^{*}\times\mathbb R$. Define the set
$A_1 = {\overline {co}^{w^{*}} }\left( {A\cup N} \right)$. Then  $A_{1}$ is a symmetric inner non-empty bounded weak$^{*}$ closed convex subset of $X^{*}\times\mathbb R$.
By the hypothesis, we obtain that $\overline{co}^{w^{*}}(w^{*}\mathrm{exp} A_{1})$ has the non-empty interior.
Hence there exists a point $(y^{*},d)\in w^{*}\mathrm{exp} A_{1}$ such that $(y^{*},d)\notin $ $  N $. Hence there exists a point  $(y,l)\in S(X\times\mathbb R)$ such that $(y^{*},d)$ is exposed by $(y,l)$.
Since  $A_1$ is  symmetric, we get that $-(y^{*},d)$ is a weak$^{*}$ exposed point of $A_1$ and is exposed by $-(y,l)\in S(X\times\mathbb R)$.
Since ${A }$ and $ N $ are two weak$^{*}$ closed subsets of $X^{*} \times R$, we obtain that
$A\cup N$ is a weak$^{*}$ closed set. Therefore, by Lemma 2.2, we conclude that $({y^{*}},d)\in  {A\cup N}$.
Therefore, since $({y^{*}},d)\in  {A\cup N}$ and $(y^{*},d)\notin N $, we obtain that $({y^{*}},d)\in  A $.

\par Since $({y^{*}},d)$ is a weak$^{*}$ exposed point of $A_1$ and $(y,l)\in S( X \times\mathbb R) $ is an exposed functional, by $A\subset A_{1}$,
we conclude that $({y^*},d)$ is a  weak$^{*}$ exposed point of $A$ and $(y,l)$ is an exposed functional. Therefore, by the Proposition 6.9 of [5], we obtain that
$(y,l)$ is a G$\mathrm{\hat{a}}$teaux
differentiable point of ${\sigma _{{A}}}$. Since $A$ is a symmetric set,
we obtain that $-(y,l)$ is a G$\mathrm{\hat{a}}$teaux
differentiable point of ${\sigma _{{A}}}$. Let $({z_0^*},h_0) \in N\subset A_{1}$. Then, by the formula
$A\subset B(X^{*}\times\mathbb R)$, we obtain that $\left\| ({y^{*}},d )\right\| \le 1$. Since $({y^*},d)$ is a weak$^{*}$ exposed point of $A$ and $(y,l)$ is a weak$^{*}$ exposed functional, by
the fact that $\left\| ({y^{*}},d )\right\| \le 1$ and $\left\|( y,l ) \right\|  = 1$, we obtain the following inequalities
\[
\left\langle \left( y,l \right), \left( ({z_0^*},h_0)  \right) \right\rangle  \le \sigma _{A_1}\left( {y,l} \right) =\langle(y^{*},d), ( y,l )\rangle\leq \left\| (y^{*},d)\right\|\cdot \left\|( y,l ) \right\|  = 1.
~~~~~~~~~~~~~~~~~~~~~~~~~~~~~~~~~~~~~~~~~~~~~~~~~~~~~~~~~~~~\eqno~~~~~~~~~~~~~~~~~~~~~~~~~~~~~~~~~~~~~~~~  (2.8)
\]
Therefore, by symmetry of $N$ and the formula (2.8), we have $\left| {\left\langle {\left( {y,l} \right), \left( {{z^{*}},{h}} \right)} \right\rangle } \right| \le 1$
for all $({z^{*}},h) \subset N\subset A_{1}$. Therefore, by Lemma 1.9, we obtain that either
\[
\left\| \left( x,r\right)  - \left( y,l\right )\right\| \leq \varepsilon ~~~\quad~~~~\mathrm{or}~~~~\quad~~~~~ \left\| ( x,r) - \left( -y,-l  \right)\right\| \leq \varepsilon .
\]
Since ${\sigma _{{A}}}$ is G$\mathrm{\hat{a}}$teaux
differentiable at the points $( y,l )$ and $( -y,-l )$, by the above formula and the arbitrariness of $( x,r)\in S\left(X\times\mathbb R\right)$, we obtain that ${\sigma _{{A}}}$
is G$\mathrm{\hat{a}}$teaux
differentiable on a dense subset $D$ of $X \times\mathbb R$.

\par  \textbf{Step 2.} Let $f$ be a continuous convex function on $X$. Pick a point $x_{0}\in X$ and a real number $\eta\in (0,1)$.
We next  prove that there exists a point $y_0\in B\left(0,\eta\right)\subset B\left(X\right)$ such that $f$ is G$\mathrm{\hat{a}}$teaux
differentiable at $y_0+x_{0}\in X$. Let
\[
g\left(x\right)=f\left(x+x_{0}\right)-f\left(x_{0}\right)-4,\quad x\in X. ~~~~~~~~~~~~  \eqno~~~~~~~~~~~~   (2.9)
\]
Since the function $f$ is a continuous convex function on $X$, by the formula (2.9), we obtain that $g$ is a continuous convex function on $X$ and $g(0)=-4$. Let $\mathrm{epi}g$ $=\left\{ ( x,r)\in X \times\mathbb R: g(x)\leq  r       \right\}$.
Define the set
\[
E \left(g, \lambda\right)= \left\{ ( x,r)\in X \times\mathbb R: g(x)\leq  r       \right\} \cap \left\{ ( x,t)\in X \times\mathbb R:  | t | \leq \lambda,~x\in  \lambda B(X)     \right\}
\]
for each $\lambda\geq 5$.
Then $E (g, \lambda)$ is an inner non-empty bounded closed convex subset of $X\times\mathbb R$ and $(0,0)\in \mathrm{int} E\left(g, \lambda\right)$. Hence we define the closed convex set
\[
D \left(g, \lambda\right)=E \left(g, \lambda\right) \cap \big(-E \left(g, \lambda\right)\big) ~~~~~~~ ~~~~~~~~~~~~~~~~~ \eqno  ~~~~~~~~~~~~~~~~~~~~~~~~ (2.10)
\]
for all $\lambda\geq 5$. Since $E (g, \lambda)$ is a inner non-empty bounded closed convex subset of $X\times \mathbb R$ and $(0,0)\in \mathrm{int} E(g, \lambda)$, by the definition of $D (g, \lambda)$,
we get that $D (g, \lambda)$ is a symmetric inner non-empty bounded closed convex set and $(0,0)\in \mathrm{int} D (g, \lambda)$.
Moreover, since $g$ is a continuous convex function on $X$,
there exists a real number $\delta\in (0,\eta)$ such that $\left|g(x)-g(0)\right|<1$
whenever $x\in B(0,\delta)$.
Pick a real number
$\lambda_{0}\in (10,+\infty)$. Then we  define the function
\[
g_{\lambda_{0}}(x) = \inf \left\{ r\in\mathbb R:(x,r)\in D \left(g, \lambda_{0}\right)\right\},\quad x\in T \left( D \left(f, \lambda_{0}\right) \right).                   \eqno                          (2.11)
\]
Therefore, by the formula $(0,0)\in \mathrm{int} D \left(g, \lambda_{0}\right)$, we obtain that $0\in T \left(\mathrm{int}  D (f, \lambda_{0})\right)$, where
$T$ is a mapping from $X \times \mathbb R $ to $  X$ satisfy $T(x,r)=x$. Hence we may assume without loss of generality that $B(0,\delta)\subset T \left(\mathrm{int} D \left(f, \lambda_{0}\right)\right)$.  Therefore, since
\[
g\left(x\right) = \inf \left\{ r\in\mathbb R:(x,r)\in \mathrm{epi}g \right\}~\quad~\mathrm{and}~\quad~ D \left(g, \lambda_{0}\right)\subset \mathrm{epi}g,
\]
we obtain that $g(x)\leq g_{\lambda_{0}}(x)$ whenever $x\in TD \left(g, \lambda_{0}\right)$.
Noting that $\left|g(x)-g(0)\right|$ $<1$ whenever $x\in B(0,\delta)$, by the fact that $g(0)=-4$, we obtain that
\[
-5=g(0)-1<g(x)<g(0)+1= - 3~\quad~\mathrm{for}~~\mathrm{every}~\quad~x\in B\left(0,\delta\right).           \eqno  (2.12)
\]
Therefore, by the definition of $E (g, \lambda_{0})$, we obtain that $(x,g(x))\in E(g, \lambda_{0})$ for all $x\in B\left(0,\delta\right)$. Moreover, by the
formula (2.12), we have $3<-g(x)<5$ for every $x\in B\left(0,\delta\right)$. Noting that $x\in B\left(0,\delta\right)$, we conclude that $-x\in B\left(0,\delta\right)$. Therefore, by the formula (2.12), we have $-5<g(-x)<-3$.
Therefore, by the definition of $E (g, \lambda_{0})$ and $\lambda_{0}>10$, we conclude that $\left(-x,-g(x)\right)\in E\left(g, \lambda_{0}\right)$. It follows that $(x,g(x))\in -E\left(g, \lambda_{0}\right)$.
Therefore, by the
formula (2.10) and $(x,g(x))\in E(g, \lambda_{0})$, we have $(x,g(x))\in D\left(g, \lambda_{0}\right)$ for every $x\in B(0,\delta)$.
Then, by the definition of $g_{\lambda_{0}}$, we have $g_{\lambda_{0}}(x)\leq g(x)$ for every  $x\in B(0,\delta)$. We have shown that
$g_{\lambda_{0}}(x)\geq g(x)$ for all $x\in B(0,\delta)$.
Hence $g(x)= g_{\lambda_{0}}(x)$ for all $x\in B\left(0,\delta\right)$.
Define the set
\[
D^{*} (g, \lambda_{0})=\left\{( x^{*},r^{*})\in X^{*}\times\mathbb R  :   \left\langle( x^{*},r^{*}), ( x,r)     \right\rangle \leq 1 ,~( x, r)   \in D \left(g, \lambda_{0}\right)  \right\}.
\]
Then we get that $D^{*} (g, \lambda_{0})$ is a symmetric inner non-empty bounded weak$^{*}$ closed convex subset of $X^{*}\times\mathbb R$. Therefore,
from the proof of Step 1, we get that ${\sigma _{D^{*} (g, \lambda_{0})}}$
is G$\mathrm{\hat{a}}$teaux
differentiable on a dense subset $G$ of $X \times\mathbb R$.

\par We next prove that there exists a point $y_0\in B(0, \delta/2)$
such that $g_{\lambda_{0}}$ is G$\mathrm{\hat{a}}$teaux
differentiable at the point $y_0$. Since ${\sigma _{D^{*} (g, \lambda_{0})}}$
is G$\mathrm{\hat{a}}$teaux
differentiable on a dense subset $G$ of $X \times\mathbb R$. Pick a point $y_0\in T\left(\mathrm{int} D \left(g, \lambda_{0}\right)\right)$ such that $y_0\in B\left(0, \delta/2\right)$ and
$\sigma _{D^{*} (g, \lambda_{0})}$ is G$\mathrm{\hat{a}}$teaux differentiable at $(y_0,g_{\lambda_{0}}(y_0))\in X \times\mathbb R$. Pick a point $y_{0}^{*}\in $ $\partial g_{\lambda_{0}}(y_0)$.
Then, for every $(z,r)\in \mathrm{int} D \left(g, \lambda_{0}\right)$, we conclude that $r\geq g_{\lambda_{0}}(z)$. Since
$y_{0}^{*}\in $ $\partial g_{\lambda_{0}}(y_0)$ and $r\geq g_{\lambda_{0}}(z)$, we obtain that
\[
y_{0}^{*}(z)-r\leq y_{0}^{*}(z)-g_{\lambda_{0}}(z)\leq y_{0}^{*}(y_{0})-g_{\lambda_{0}}(y_{0})~\quad~~~~~~~~~\mathrm{for}~~~~~~~\mathrm{every}~~~~~~~~~~~\quad~~(z,r)\in \mathrm{int} D \left(g, \lambda_{0}\right).
\]
Hence we obtain that $\left\langle({z_{0}^{*}},r_{0}),(y_{0},g_{\lambda_{0}}(y_{0}))\right\rangle =1$ and $({z_{0}^{*}},r_0)\in D^{*} (g, \lambda_{0})$,
where
\[
z_{0}^{*}=   \frac{y_{0}^{*}}{{y_{0}^{*}}\left(y_{0}\right) - g_{\lambda_{0}}\left(y_{0}\right)}~~\quad~~~~~~~~~\mathrm{and}~~~~~~~~~~~~\quad~~~r_{0}= \frac{-1}{{y_{0}^{*}}\left(y_{0}\right) - g_{\lambda_{0}}\left(y_{0}\right)}. ~~~~~~~~~~~~~~~~~ \eqno  ~~~~~~~~~~~~~~~~~~~~~~~~ (2.13)
\]
Hence we obtain that $\left({z_{0}^{*}},r_{0}\right)\in \partial \sigma_{D^{*} (g, \lambda_{0})}\left(y_{0},g_{\lambda_{0}}\left(y_{0}\right)\right)$. It follows that $y_0^{*}=-r_0z_{0}^{*}$.
Since ${\sigma _{D^{*} (g, \lambda_{0})}}$
is G$\mathrm{\hat{a}}$teaux differentiable at $(y_0,g_{\lambda_{0}}(y_0))$, we conclude that
 $g_{\lambda_{0}}$ is G$\mathrm{\hat{a}}$teaux  differentiable at $y_0\in X$.

\par Since $g(x)= g_{\lambda_{0}}(x)$ for every  $x\in B\left(0,\delta\right)$, we conclude that  $g$ is G$\mathrm{\hat{a}}$teaux
differentiable at  $y_0\in X$. Therefore, by the formula (2.9), we obtain that
$f$ is G$\mathrm{\hat{a}}$teaux
differentiable at $y_0+x_{0}$. Hence we conclude that $f$ is G$\mathrm{\hat{a}}$teaux  differentiable on a dense subset of $X$. Hence
we obtain that $X$ is a G$\mathrm{\hat{a}}$teaux  differentiability space, which completes the proof.
\end{proof}

We next will prove that Theorem 2.1.

\begin{proof}
Let $C^*$  denote a symmetric bounded weak\(^{*}\) closed convex subset of \(M^*\) with nonempty interior. Then
$0\in \mathrm{int}\left(C^*\right)$.
Since $C^*$ is bounded,
we may assume without loss of generality that $C^*\subset B\left(M^{*}\right)$.
We next prove that $\sigma_{C^*}$ is G$\mathrm{\hat{a}}$teaux differentiable on a dense subset of $M$.
In fact, pick a point $x\in S(M)$ and a real number $r\in (0,1)$. We only need to prove that there exists a point $y\in B(x,r)\cap M$ such that
$\sigma_{C^*}$ is G$\mathrm{\hat{a}}$teaux differentiable at $y\in M$.
Since $ 0\in \mathrm{int}{C^*}$, there exists a real number $a\in (0,1)$ such that $\{x^{*}\in M^{*}: \|x^{*}\|\leq a\}\subset \mathrm{int}{C^*}$.
Therefore, since $x\in S\left(M\right)$ and $\{x^{*}\in M^{*}: \|x^{*}\|\leq a\}\subset \mathrm{int}{C^*}$, there exists a point $z^{*}\in {C^*} $ such that $z^{*}(x)= a$.
We next use the weak$^{*}$ slice iterative selection method to derive
a contradiction.
To explain it clearly, we next  divide
the proof into four steps.

\textbf{Step 1.}
Let $y_{1}=x$ and $\beta_{1}=\min\left\{ r/256,a/256\right\}$. Then we obtain that $\beta_{1}>0$.
Define the weak$^{*}$ bounded closed convex subset $N_{1}^{*}$ of $M^*$, where
\[
N_{1}^{*}=\left\{ {x^*} \in M^*:\left\| x^* \right\| \leq \left(\frac{64}{\beta_{1}\cdot 2^{-1}}\right)  \sup_{y^* \in C^*} \left\| {{y^*}} \right\|,~x^*(x) ={x^*}(y_1) = 0\right\} .
\]
Let $I: M\rightarrow X$ denote the identity operator. Then we get that $I^{*}: X^{*}\rightarrow M^{*}\cong {{{X^*}} \mathord{\left/
 {\vphantom {{{X^*}} {{M^{\bot}}}}} \right.
 \kern-\nulldelimiterspace} {{M^{\bot}}}}$ and $\|I\|=\|I^{*}\|=1$, where the mapping $I^{*}$ denotes adjoint operator of $I$.
 It is easy to see that $I^{*}$ is weak$^{*}$-to-weak$^{*}$ continuous and the range of  $I^{*}$ is $M^{*}$.
Moreover, by the open mapping theorem, we get that $I^{*}$ is an open mapping.
Hence we obtain that $I^{*}\left(B\left(0,1\right)\right)$ contains a neighborhood of the origin in $M^{*}$.

\par Since $C^{*}$ is a bounded subset of $B(M^{*})$, by the Hahn-Banach theorem, we get that for any $x^{*}\in C^{*}$, there exists a functional $y^{*}\in B(X^{*})$
such that $I^{*}(y^{*}) =x^{*}$. Let $ K_{\lambda}= B(X^{*})\cap (I^{*})^{-1}(C^{*})$. Then we get that
 $I^{*}(K_{\lambda})=C^{*}$.

\par We define $J=\{K':K'$ is a weak$^{*}$ closed, convex subset of $B(X^{*})$ with $0\in K'$ and $I^{*}(K')=C^{*}\}$. Therefore, by the containing relation, we define the order set $J$, i.e,
$A\leq B$ if and only if $A\subset B$.
Let $J_0= \{K_{\alpha}^{*}\in J: \alpha\in\Delta\}$ denote a totally ordered subset. We claim that
$I^{*}\left(\cap_{\alpha\in\Delta} K_{\alpha}^{*}\right)= C^{*}$. In fact, pick a point $x^{*}\in C^{*}$. Then, since $K_{\alpha}^{*}$ is a  weak$^{*}$ closed convex set, we conclude that $(I^{*})^{-1}(x^{*})\cap K_{\alpha}^{*}$ is a weak$^{*}$ compact convex set. Then, by the finite intersection property of compact sets, we conclude that
\[
\mathop\cap\limits_{\alpha\in\Delta} \left( K_{\alpha}^{*}\cap (I^{*})^{-1}(x^{*}) \right)\neq \emptyset. ~~~~~~~~~~~~~~~~~ \eqno  ~~~~~~~~~~~~~~~~~~~~~~~~ (2.14)
\]
It follows that  $I^{*}\left(\cap_{\alpha\in\Delta} \left( K_{\alpha}^{*}\cap (I^{*})^{-1}(x^{*}) \right)\right)=x^{*}$.
Since $x^{*}\in C^{*}$ is arbitrary,
we conclude that $I^{*}\left(\cap_{\alpha\in\Delta} K_{\alpha}^{*}\right)\supset C^{*}$. Then, since $I^{*}\left(\cap_{\alpha\in\Delta} K_{\alpha}^{*}\right)\subset C^{*}$, we obtain that  $I^{*}\left(\cap_{\alpha\in\Delta} K_{\alpha}^{*}\right)= C^{*}$.
It follows that
$\cap_{\alpha\in\Delta} K_{\alpha}^{*}\in J$.
Then, by the Zorn lemma, there exists a minimal element $K_{1}^{*}\in J$ such that  $I^{*}(K_{1}^{*})=C^{*}$.
Let
\[
H_{1}^{*}=K_{1}^{*}\cap (I^{*})^{-1}(N_{1}^{*}).
\]
Since $0\in \mathrm{int}C^*$, we obtain that $N_{1}^{*}\cap {C^*} \neq \emptyset$. Since $K_{1}^{*}$ is weak$^{*}$ closed and $H_{1}^{*}= $ $K_{1}^{*}\cap (I^{*})^{-1}(N_{1}^{*})$, we obtain that
$H_{1}^{*}$ is a nonempty weak$^{*}$ closed convex set.
Let
\[
X_{1}^{*}= \left\{  \lambda x^{*} :    x^{*}\in  K_{1}^{*},~x^{*}|_{M}\notin \mathrm{int}|_{M^{*}}(C^{*}),~ \lambda\in \mathbb R^{+}           \right\}.
\]
Then, by the definition of $X_{1}^{*}$, we conclude that $C^{*}\subset I^{*}(X_{1}^{*})$. We claim that
\[
K_{1}^{*}\supset X_{1}^{*}\cap (I^{*})^{-1}(C^{*}). ~~~~~~~~~~~~~~~~~ \eqno  ~~~~~~~~~~~~~~~~~~~~~~~~ (2.15)
\]
In fact, suppose that there exists a point $z_{0}^{*}\in X_{1}^{*}\backslash K_{1}^{*}$ so that $I^{*}(z_{0}^{*})\in C^{*}$.
Then,
since $K_{1}^{*}$ is a bounded weak$^{*}$ closed convex subset of $X^{*}$, $0\in \mathrm{int} C^*$ and $0\in K_{1}^{*}$, there exists a real number $t_0\in [0,1]$
such that
\[
I^{*}(t_0 z_{0}^{*})\in \mathrm{int}|_{M^{*}}(C^{*})~~~\quad~~~~~~~ \mathrm{and}~~~~~\quad~~~~~~~~ t_0 z_{0}^{*}\in X_{1}^{*}\backslash K_{1}^{*}.
\]
Hence we may assume without loss of generality that
$I^{*}(z_{0}^{*})\in \mathrm{int}|_{M^{*}}(C^{*})$.
Noting that $z_{0}^{*}\in X_{1}^{*}$, by the definition of $X_{1}^{*}$, there exist a real number $\lambda\in [0, +\infty)$ and
a point $y_{0}^{*} \in  K_{1}^{*}$ with $y_{0}^{*}|_{M}\notin \mathrm{int}|_{M^{*}}(C^{*})$ so that $y_{0}^{*}=\lambda z_{0}^{*}$.
We claim that $\lambda\in [0,1]$. In fact, suppose that $\lambda\in (1,+\infty)$.
Since $K_{1}^{*}$ is a weak$^{*}$ closed convex subset of $X^{*}$,  $y_{0}^{*}=\lambda z_{0}^{*}$ and $\lambda\in (1,+\infty)$, we get that $z_{0}^{*}\in \left[0, y_{0}^{*}\right]$.
Therefore, since $y_{0}^{*} \in  K_{1}^{*}$ and $0\in K_{1}^{*}$, we have $z_{0}^{*}\in K_{1}^{*}$,  a contradiction. Hence $\lambda\in [0,1]$.
Therefore, since $y_{0}^{*}=\lambda z_{0}^{*}$
and
$I^{*}(z_{0}^{*})\in \mathrm{int}|_{M^{*}}(C^{*})$, we obtain that
\[
y_{0}^{*}|_{M}=I^{*}\left(y_{0}^{*}\right)=I^{*}\left(\lambda z_{0}^{*}\right)=\lambda I^{*}\left( z_{0}^{*}\right)\in \mathrm{int}|_{M^{*}}\left(C^{*}\right),
\]
a contradiction. Hence we obtain that $K_{1}^{*}\supset X_{1}^{*}\cap (I^{*})^{-1}(C^{*})$.
Define the set
\[
M_{1}^{*}= \left\{  \lambda x^{*} :    x^{*}\in  H_{1}^{*},~x^{*}|_{M}\notin \mathrm{int}|_{M^{*}}(C^{*}),~ \lambda\in  \mathbb R^{+}           \right\}.
\]
Then we have $M_{1}^{*}\subset X_{1}^{*}$. Noting that $0\in \mathrm{int}\left(C^*\right)$, we define the nonempty set
\[
L_{1}^{*}=M_{1}^{*}\cap (I^{*})^{-1}(N_{1}^{*}).
\]
Thus $I^{*}(L_{1}^{*})=N_{1}^{*}$.
Since $K_{1}^{*}\supset X_{1}^{*}\cap (I^{*})^{-1}\left(C^{*}\right)$, we get that $L_{1}^{*}$ is bounded. Let
\[
D_{1}^{*}=\overline{co}^{w^{*}} \left(K_{1}^{*}\cup L_{1}^{*}\right)~~\quad~~~~~~~~~~~~~~\mathrm{and}~~~~~~~~~~~~~~~~~~\quad~~~~~~~~D_{1,0}^{*}=\overline{co}^{w^{*}} \left(C^{*}\cup N_{1}^{*}\right).
\]
Then  $D_{1}^{*}$ is a bounded weak$^{*}$ closed convex subset of $X^{*}$ and $D_{1,0}^{*}$ is a bounded weak$^{*}$ closed convex subset of $M^{*}$.
Moreover, we have $I^{*}(D_{1}^{*})=D_{1,0}^{*}$.

\par Let $\gamma_0= \sup \left\{\|x^{*}\|:x^{*}\in L_{1}^{*}\right\}$. Then, since  $L_{1}^{*}$ is a bounded set,
we obtain that $\gamma_0\in (0,+\infty)$.
Since  $K_{1}^{*}$ is a bounded subset of $X^{*}$ and $ \mathrm{int}{C^*}\neq \emptyset$, there exists a real number $\theta\in (0,+\infty)$ such that
$\theta=\sup \left\{\|x^{*}\|:x^{*}\in K_{1}^{*}\right\}$.
Moreover, since $X$ is a G$\mathrm{\hat{a}}$teaux differentiability space, there exists a point $x_1\in S(X)\backslash M$
such that

\par (1) the functional $\sigma_{K_{1}^{*}}$ is G$\mathrm{\mathrm{\hat{a}}}$teaux differentiable at $x_1\in X$;
\par (2) $\|x-x_1\|<256^{-1}\beta_1(\gamma_0+1)^{-1}(\theta+1)^{-1}$.
\\
Pick a point $x^{*}\in L_{1}^{*}$. Then, by condition (2) and $x^{*}(x)=0$, we conclude that
\[
x^{*}(x_1)=x^{*}(x_1-x)+x^{*}(x)\leq \|x^{*}\|\cdot \|x-x_1\|\leq \gamma_0 \cdot \frac{\beta_1}{256(\gamma_0+1)(\theta+1)}\leq \frac{\beta_1}{256}.
\]
Noting that $I^{*}(K_{1}^{*})=C^{*}$ and $z^*\in C^*$, there exists a point $x_0^{*}\in K_{1}^{*}$ such that $I^{*}\left(x_0^{*}\right)=z^*$. Therefore, since $K_{1}^{*}\subset B(X^{*})$, we have $\|x_0^{*}\|\leq 1$ and $x_0^{*}(x)=z^*(x) $.
Therefore, since $z^*(x)=a$ and $\beta_{1}=\min\left\{ r/256,a/256\right\}$, we obtain that
\[
x_0^{*}\left(x_1\right)~=~ x_0^{*}\left(x\right)- x_0^{*}\left(x-x_1\right)\quad\quad \quad\quad \quad
\]
\[
\quad\quad \quad\quad \quad\quad \quad\quad \quad\quad \quad\quad =~ z^*\left(x\right)~-~ x_0^{*}\left(x-x_1\right)\quad \quad\quad \quad\quad \quad\quad \quad\quad \quad\quad \quad\quad \quad\quad \quad\quad \quad\quad \quad\quad \quad\quad \quad\quad\quad \quad\quad \quad\quad
\]
\[
\quad\quad \quad\quad \quad\quad \quad\quad \quad\quad \quad\quad \geq ~z^*\left(x\right)~-~\left\|x_0^{*}\right\|\cdot \left\|x-x_1\right\|\quad \quad\quad \quad\quad \quad\quad \quad\quad \quad\quad \quad\quad \quad\quad \quad\quad \quad\quad \quad\quad \quad\quad \quad\quad\quad \quad\quad \quad\quad
\]
\[
\quad\quad \quad\quad \quad\quad \quad\quad \quad\quad \quad\quad \geq~  a~-~\left(\theta+1\right)\frac{\beta_1}{256(\gamma_0+1)(\theta+1)} ~>~ \frac{1}{2}a. \quad \quad\quad \quad\quad \quad\quad \quad\quad \quad\quad \quad\quad \quad\quad \quad\quad \quad\quad \quad\quad \quad\quad \quad\quad\quad \quad\quad \quad\quad
\]
Therefore, from the above inequalities, we conclude that $\sigma_{K_{1}^{*}}\left(x_1\right)>\sigma_{L_{1}^{*}}\left(x_1\right) $.
We pick a real number $\eta_1\in (0,\beta_1/256)$. Since $K_{1}^{*}$ is a bounded weak$^{*}$ closed convex subset of $X^{*}$, $L_{1}^{*}$ is a bounded subset of $X^{*}$,
and $\sigma_{K_{1}^{*}}(x_1)>\sigma_{L_{1}^{*}}(x_1) $, by Lemma 2.3, there exists a real number $\alpha_1\in  (0,\eta_1/2)$ such that
\[
S\left(x_1, D_{1}^{*},\alpha_1\right)\subset S\left(x_1, K_{1}^{*},\frac{1}{2}\eta_1\right)+ B\left(0, \frac{1}{2}\eta_1\right).~~~\eqno~~~~~~~~~~~~(2.16)
\]
Since $D_{1}^{*}$ is a weak$^{*}$ bounded closed convex subset of $X^{*}$, we conclude that
\[
D_{1}^{*}\backslash S\left(x_1,\alpha_1, D_{1}^{*}\right)=D_{1}^{*}\cap\left\{x^{*}\in X^{*}:x^{*}(x)\leq \sigma_{ D_{1}^{*}}(x)-\alpha_1\right\}~
\]
is a weak$^{*}$ closed convex set. Moreover, by the minimality of $ K_{1}^{*}$, we obtain that
\[
A^{*}=  I^{*}\left( K_{1}^{*}\backslash S\left(x_1, \alpha_{1}, K_{1}^{*} \right)\right)\neq I^{*}\left( K_{1}^{*}\right).~~~~~~~~~~~~~\eqno~~~~~~~~~(2.17)
\]
Hence we obtain that $C^{*}\backslash A^{*}\neq \emptyset$. We pick a point $y^{*}\in C^{*}\backslash A^{*}$ and define the set
\[
F^{*}= I^{*}\left( D_{1}^{*}\backslash S\left(x_1, \alpha_{1}, D_{1}^{*}\right)\right).~~~~~~~~~~~~~\eqno~~~~~~~~~(2.18)
\]

\par
We claim that $y^{*}\in D_{1,0}^{*}\backslash F^{*}$. In fact, suppose that there exists a point $x^{*}\in D_{1}^{*}\backslash  S\left(x_1,\alpha_{1}, D_{1}^{*}\right)$ such that
$y^{*}= I^{*}(x^*)$. Since $K_{1}^{*}\supset X_{1}^{*}\cap (I^{*})^{-1}(C^{*})$, $y^{*}\in C^{*}\backslash A^{*}$
and $x^{*}\in D_{1}^{*}$, by the definition of $D_{1}^{*}$, we conclude that $x^{*}\in K_{1}^{*}$. Moreover, since
$x^{*}\in D_{1}^{*}\backslash  S\left(x_1,\alpha_{1}, D_{1}^{*}\right)$ and $\sigma_{K_{1}^{*}}(x_1)>\sigma_{L_{1}^{*}}(x_1) $, we have $x^{*}\in K_{1}^{*}\backslash S\left(x_1,\alpha_{1}, K_{1}^{*}\right)$.
Therefore, by the formula (2.17), we conclude that
$y^{*}= I^{*}(x^*)\in A^{*}$. However, since
$y^{*}\in C^{*}\backslash A^{*}$, we obtain that $y^{*}\notin A^{*}$, a contradiction. Hence we obtain that
$y^{*}\in D_{1,0}^{*}\backslash F^{*}$. This implies that
$D_{1,0}^{*}\backslash F^{*}\neq \emptyset$.

\par Since $D_{1,0}^{*}\backslash F^{*}\neq \emptyset$, we pick a point
$x_{2}^{*}\in D_{1,0}^{*}\backslash F^{*}$. Then $x_{2}^{*}\notin   F^{*}$.
Noting that
\[
D_{1}^{*}\backslash S\left(x_1,\alpha_1, D_{1}^{*}\right)=D_{1}^{*}\cap\left\{x^{*}\in X^{*}:x^{*}(x)\leq \sigma_{ D_{1}^{*}}(x)-\alpha_1\right\}
\]
is a weak$^{*}$ bounded closed convex subset of $X^{*}$, we obtain that
$D_{1}^{*}\backslash S\left(x_1, \alpha_{1}, D_{1}^{*}\right)$ is a weak$^{*}$ compact convex subset $X^{*}$. Since
$I^{*}$ is a weak$^{*}$-to-weak$^{*}$ continuous mapping and
$F^{*}= I^{*}\left( D_{1}^{*}\backslash S\left(x_1, \alpha_{1}, D_{1}^{*}\right)\right)$, we conclude that $F^{*}$ is a weak$^{*}$ compact convex subset of $M^{*}$.
Then, since $x_{2}^{*}\notin   F^{*}$, by the separation theorem, there exists a point $y_{2}\in S\left(M\right)$ and
a real number $\beta_{2}\in \left(0,\beta_1/256 \right)$ such that
\[
x_2^*\left(y_2\right) - 4 \beta_{2} > \sup \left\{ {x^*}({y_2}):{x^*} \in F^*\right\}.~~~~~~~~~~~~~~~~~~~~~~~~~\eqno~~~~~~~~~(2.19)
\]
Pick a point $x^*\in S\left(y_2 ,D_{1,0}^{*},\beta_{2}\right) $. Then, by the formula (2.19), we conclude that
\[
{x^*}\left(y_2\right)\geq {\sigma _{D_{1,0}^{*}}}({y_2}) - \beta_{2} > x_2^*\left(y_2\right) - 4\beta_{2} > \sup \left\{ {x^*}\left(y_2\right):{x^*} \in F^* \right \}.
\]
This implies that ${x^*} \not\in F^{*}$. Then, since ${x^*} \in D_{1,0}^{*}$, we conclude that ${x^*} \in D_{1,0}^{*}\backslash F^{*}$.
Since $x^*\in S\left(y_2 ,D_{1,0}^{*},\beta_{2}\right) $ is  arbitrary, we conclude that
\[
S\left({y_2},D_{1,0}^{*},\beta_{2}\right) = \left\{ {x^*} \in D_{1,0}^{*}:{x^*}({y_2}) \geq {\sigma _{D_{1,0}^{*}}}({y_2}) - \beta_{2}\right\} \subset D_{1,0}^{*}\backslash F^{*}.
\]
Therefore, since $I^{*}(D_{1}^{*})=D_{1,0}^{*}$ and $F^{*}=I^{*}\left( D_{1}^{*}\setminus S\left(x_1,\alpha_{1}, D_{1}^{*}\right)\right)$, we obtain that
\[
S\left({y_2},D_{1,0}^{*},\beta_{2}\right)\subset  D_{1,0}^{*}\backslash I^{*}\left( D_{1}^{*}\setminus S\left(x_1,\alpha_{1}, D_{1}^{*}\right)\right)\subset I^{*}\left( S\left(x_1,\alpha_{1}, D_{1}^{*}\right)\right).
\]
Pick a point $y^{*} \in S({y_2},D_{1,0}^{*},\beta_{2})$. Then, by  the formula (2.16) and $\eta_1\in (0,\beta_1/256)$, there exists a point
$x^{*} \in S(x_1,\alpha_{1}, D_{1}^{*})$ with $I^{*}(x^{*})=y^{*}$ such that $\|x^{*}\|\leq 2$.
Since $z^{*}(x)=a$ and $\|x-x_1\|<256^{-1}\beta_1(\gamma_0+1)^{-1}(\theta+1)^{-1}$, by the fact that $I^{*}(x_0^{*})=z^*$ and $\|x_0^{*}\|\leq 1$, we obtain the following inequalities
\[
~y^{*} \left(x\right)=x^{*} \left(x\right)~=~x^{*} \left(x_1\right)-x^{*} \left(x_1-x\right)\quad \quad\quad \quad\quad \quad\quad \quad
\]
\[
  \quad\quad\quad \quad \quad \quad\quad \quad\quad \quad\quad \quad\geq~ \sigma_{D_{1}^{*}}\left(x_1\right)~-~\alpha_{1}~-~x^{*} \left(x_1-x\right)\quad \quad \quad \quad\quad \quad \quad \quad\quad \quad \quad \quad\quad \quad \quad \quad\quad \quad \quad \quad\quad \quad \quad \quad\quad\quad \quad \quad \quad
\]
\[
 \quad \quad\quad \quad\quad \quad\quad \quad \quad \quad\quad \quad\geq~ x_0^{*}\left(x_1\right)~-~x^{*} \left(x_1-x\right)~-~\alpha_{1}\quad \quad \quad \quad\quad \quad \quad \quad\quad \quad \quad \quad\quad \quad \quad \quad\quad \quad \quad \quad\quad \quad \quad \quad\quad\quad \quad \quad \quad
\]
\[
  \quad\quad\quad \quad \quad\quad \quad\quad \quad \quad\quad \quad\geq~ x_0^{*}\left(x\right)~- ~ x_0^{*} \left(x-x_1\right)      ~ - ~x^{*} \left(x_1-x\right)~-~\alpha_{1}\quad \quad \quad \quad\quad \quad \quad \quad\quad \quad \quad \quad\quad \quad \quad \quad\quad \quad \quad \quad\quad \quad \quad \quad\quad\quad \quad \quad \quad
\]
\[
  \quad\quad\quad\quad \quad\quad \quad \quad \quad \quad\quad \quad\geq~ z^{*}\left(x\right)~-~  \left\|x_0^{*}\right\| \left\|x_1-x\right\|     ~ -~\left\|x^{*} \right\|\left\|x_1-x\right\|~-~\alpha_{1}\quad \quad \quad \quad\quad \quad \quad \quad\quad \quad \quad \quad\quad \quad \quad \quad\quad \quad \quad \quad\quad \quad \quad \quad
\]
\[
 \quad \quad\quad \quad \quad\quad \quad\quad \quad \quad\quad \quad\geq~ a~-~\frac{1}{256}\beta_1~-~\frac{2}{256}\beta_1~-~\alpha_{1}>\frac{1}{2}a. \quad \quad \quad \quad\quad \quad \quad \quad\quad \quad \quad \quad\quad \quad \quad \quad\quad \quad \quad \quad\quad \quad \quad \quad
\]
Since $y^{*} \in S\left({y_2},D_{1,0}^{*},\beta_{2}\right)$ is  arbitrary, from the above inequalities, we obtain that
\[
S\left({y_2},D_{1,0}^{*},\beta_{2}\right)\subset  \left\{ x^* \in M^*:{x^*}(x) > 0 \right\}.~~~~~~~~~~~~~~~~~~~~~~~~~~~~~~~~~~~~~~~~~~~\eqno~~~~~~~~~~~~~~~~~~~~~~~~~~~~~~~~~(2.20)
\]
Moreover, since $S\left(x_1,\alpha_1, D_{1}^{*}\right)\subset S\left(x_1, D_{1}^{*},\alpha_1\right)$,
by the formula (2.16), we have
\begin{eqnarray*}
&&S\left(x_1,\alpha_1, D_{1}^{*}\right)-S\left(x_1,\alpha_1, D_{1}^{*}\right)
\\
&\subset& S\left(x_1, D_{1}^{*},\alpha_1\right)-S\left(x_1, D_{1}^{*},\alpha_1\right)
\\
&\subset& \left[ S\left(x_1, K_{1}^{*},\frac{1}{2}\eta_1\right)-S\left(x_1, K_{1}^{*},\frac{1}{2}\eta_1\right)+ B\left(0, \frac{1}{2}\eta_1\right)-B\left(0, \frac{1}{2}\eta_1\right) \right].
\end{eqnarray*}
Since $I^{*}$ is a bounded linear operator and $S({y_2},D_{1,0}^{*},\beta_{2})\subset   I^{*}( S(x_1,\alpha_{1}, D_{1}^{*}))$,
from the above formula,  we have the following formula
\begin{eqnarray*}
&&S\left({y_2},D_{1,0}^{*},\beta_{2}\right)-S\left({y_2},D_{1,0}^{*},\beta_{2}\right)
\\
&\subset& I^{*}\left( S\left(x_1,\alpha_{1}, D_{1}^{*}\right)\right)-I^{*}\left( S\left(x_1,\alpha_{1}, D_{1}^{*}\right)\right)
\\
&\subset& I^{*}\left( S\left(x_1, D_{1}^{*},\alpha_{1}\right)\right)-I^{*}\left( S\left(x_1, D_{1}^{*},\alpha_{1}\right)\right)
\\
&\subset& I^{*}\left( S\left(x_1, D_{1}^{*},\alpha_{1}\right)- S\left(x_1, D_{1}^{*},\alpha_{1}\right)\right)=I^{*}\left( V_1\right),
\end{eqnarray*}
where
\[
V_1=\left[ S\left(x_1, K_{1}^{*},\frac{1}{2}\eta_1\right)-S\left(x_1, K_{1}^{*},\frac{1}{2}\eta_1\right)+ B\left(0, \eta_1\right) \right].
\]
Since $D_{1,0}^{*}=\overline{co}^{w^{*}} \left(C^{*}\cup N_{1}^{*}\right)$, by Lemma 2.4 and the formula (2.20), we get that
\[
S\left({y_2},C^*,\beta_{2}\right)-S\left({y_2},C^*,\beta_{2}\right)\subset I^{*}\left( S\left(x_1, D_{1}^{*},\alpha_{1}\right)- S\left(x_1, D_{1}^{*},\alpha_{1}\right)\right)\subset I^{*}(V_1) ,
\]
\[
S\left({y_2},C^*,\beta_{2}\right)\subset \{ {x^*} \in {M^*}:{x^*}(x) > 0\}~~\quad~~~~~~\mathrm{and}~~~~~~~~~\quad~~~~~~~~~\left\|y_1-y_2\right\|<\frac{1}{4}\beta _1.~~~~\eqno~~~~~~~~~(2.21)
\]
We claim that $S(y_2,{C^*},\beta_{2})\subset  S(y_1,{C^*},\beta_{1}) $. In fact, pick a point $y^{*}\in S(y_2,{C^*},\beta_{2})$.
Then we get that
$y^{*}(y_2)\geq \sigma_{{C^*}}\left(y_2\right)-\beta_{2}$.
Therefore, since $C^* \subset B\left(M^{*}\right)$  and $y^{*}\in $ $S\left(y_2,{C^*},\beta_{2}\right)$, we obtain that $\left\|y^{*}\right\|\leq 1$.
Noting that  $\beta_{2}\in \left(0,\beta_1/256\right)$ and $y^{*}(y_2)\geq $ $\sigma_{{C^*}}(y_2)-\beta_{2}$, by
the formula (2.21), we obtain that
\[
y^{*}(y_1)~=~y^{*}\left(y_2\right)-y^{*}\left(y_2-y_1\right)\quad\quad \quad \quad \quad \quad\quad \quad \quad\quad \quad \quad\quad \quad \quad\quad
\]
\[
\quad \quad \quad \quad \quad\quad\geq~ y^{*}(y_2)~-~\left\|y^{*}\right\|\cdot\left\|y_2-y_1\right\|\quad \quad\quad \quad \quad \quad \quad \quad\quad \quad \quad\quad \quad \quad \quad \quad \quad \quad \quad \quad \quad \quad \quad \quad \quad \quad \quad  \quad \quad\quad \quad \quad \quad \quad
\]
\[
\quad \quad \quad \quad \quad\quad\geq~ \sigma_{{C^*}}(y_2)~-~\beta_{2}~-~\left\|y^{*}\right\|\cdot\left\|y_2-y_1\right\|\quad \quad \quad \quad\quad \quad \quad \quad \quad \quad \quad \quad \quad \quad \quad \quad \quad \quad \quad \quad  \quad \quad\quad \quad \quad \quad \quad
\]
\[
 \quad \quad \quad \quad \quad\quad\geq ~\sigma_{{C^*}}(y_1)~- ~\sigma_{{C^*}}(y_1-y_2) ~-~\beta_{2}~-~\left\|y^{*}\right\|\cdot\left\|y_2-y_1\right\|\quad \quad \quad \quad\quad \quad \quad \quad \quad \quad \quad \quad \quad \quad \quad \quad \quad \quad \quad \quad  \quad\quad\quad \quad \quad \quad \quad \quad \quad\quad \quad \quad \quad \quad
\]
\[
 \quad \quad \quad \quad \quad\quad = ~ \sigma_{{C^*}}(y_1)~-~\sup \left\{ u^{*}(y_2-y_1): u^{*}\in C^{*}    \right \}~ -~\beta_{2}~-~\left\|y^{*}\right\|\cdot\left\|y_2-y_1\right\|\quad \quad \quad \quad\quad \quad \quad \quad \quad \quad \quad \quad \quad \quad \quad \quad \quad \quad \quad \quad  \quad \quad\quad \quad \quad \quad \quad \quad \quad\quad \quad \quad \quad \quad
\]
\[
 \quad \quad \quad \quad \quad\quad\geq ~\sigma_{{C^*}}(y_1)~-~\left\|y_2-y_1\right\| ~-~\beta_{2}~-~\left\|y_2-y_1\right\|\quad \quad \quad \quad\quad \quad \quad \quad \quad \quad \quad \quad \quad \quad \quad \quad \quad \quad \quad \quad  \quad \quad\quad \quad \quad \quad \quad \quad \quad\quad \quad \quad \quad \quad
\]
\[
\quad \quad \quad \quad \quad\quad\geq ~\sigma_{{C^*}}(y_1)~-~\frac{1}{2}\beta_{1} ~-~\beta_{2}\quad \quad \quad \quad\quad \quad \quad \quad \quad \quad \quad \quad \quad \quad \quad \quad \quad \quad \quad \quad  \quad\quad \quad \quad \quad \quad \quad\quad \quad\quad \quad \quad \quad \quad
\]
\[
 \quad \quad \quad \quad \quad\quad\geq ~\sigma_{{C^*}}(y_1)~-~\beta_{1}.\quad \quad \quad \quad\quad \quad \quad \quad \quad \quad \quad \quad \quad \quad \quad \quad \quad \quad \quad \quad  \quad \quad\quad \quad \quad \quad \quad
\]
It follows that $y^{*}\in  S(y_1,{C^*},\beta_{1}) $. Hence we get that $S(y_2,{C^*},\beta_{2})\subset  S(y_1,{C^*},\beta_{1}) $.

\textbf{Step 2.} We  define the weak$^{*}$ bounded closed convex subset $N_{2}^{*}$ of $M^*$, where
\[
N_{2}^{*}=\left\{ {x^*} \in {M^*}:\left\| {{x^*}} \right\| \leq \left(\frac{64}{\beta_{2}\cdot 2^{-2}}\right)  \sup_{{x^*} \in {C^*}} \left\| {{x^*}} \right\|,~{x^*}(y_2) = 0\right\} .
\]
Define the set $C_{2}^{*}=\{x^{*}\in C^{*}:x^{*}(y_2)\leq 0 \}$. Since $C^{*}$ is weak$^{*}$ closed, we obtain that
$C_{2}^{*}$ is a
weak$^{*}$ closed convex subset of $M^{*}$. We next  prove that there exists a bounded weak$^{*}$ closed convex subset of $K_{2,2}^{*}$ of $ X^{*}$  with $I^{*}(K_{2,2}^{*})= C_{2}^{*}$ such that
\[
\big \langle y^{*}, x_1\big\rangle=0~~~\quad~~~~~~~~~\mathrm{whenever}~~~~~~~~~~~~\quad~~y^{*}\in K_{2,2}^{*}.
\]
Define the subspace $M(x_1)= \mathrm{span}\left \{x_1, M \right\}$ of $X$. We claim that $M(x_1)$ is a closed subspace of $X$. In fact,
we know that $x_1\in  S\left(X\right)$ and $x_1\not\in  M$. Moreover, noting that
$M$ is a closed subspace of $X$ and  $x_1\not\in  M$, by the Hahn-Banach theorem, we obtain that there
exists a functional $f\in S(X^{*})$ such that
\[
f(x_1)=\mathrm{dist}\left(x_1, M\right)>0~~\quad~~~~~~~~~~\mathrm{and}~~~~~~~~~~~~~\quad~~~~~M\subset N\left(f\right)=\left\{ x\in X: f(x)=0              \right\}.
\]
Therefore, since $M(x_1)= \mathrm{span} \{x_1, M\}$, it is easy to see that there exists a point $y_x$ $\in M$  and a real number $\alpha_x\in \mathbb R$ such that
\[
x=\alpha_x x_1 +y_{x}~\quad~~~~~~~~~~\mathrm{whenever}~~~~~~~~~~~~~\quad~~x\in M\left(x_1\right).
\]
We pick a sequence $\{x_n\}_{n=1}^{\infty}\subset M(x_1)$ such that $\left\|x_n - x_0\right\| \to 0$ as $n \to \infty$. Then there exist two sequence $\{\alpha_n\}_{n=1}^{\infty}\subset \mathbb R$ and $\{y_n\}_{n=1}^{\infty}\subset M$ such that
\[
x_n=\alpha_{n} x_1 +y_n~\quad~~~~~~~~~~\mathrm{for}~~~~~~~~~~~~\mathrm{every}~~~~~~~~~~~~~~\quad~~n\in\mathbb  N.
\]
Since $\{y_n\}_{n=1}^{\infty}\subset M$, we obtain that $f(x_n)= \alpha_{n} f(x_1)$. Since $\{x_n\}_{n=1}^{\infty}$ is a Cauchy sequence, we get that $\{f(x_n)\}_{n=1}^{\infty}$ is a Cauchy sequence.
It follows that $\{\alpha_n\}_{n=1}^{\infty}$ is a Cauchy sequence. Let $\alpha_n \to \alpha_0\in R$ as $n \to \infty$. Since $\{x_n\}_{n=1}^{\infty}$ is a Cauchy sequence and $x_n=\alpha_{n} x_1 +y_n$,
we obtain that $\{y_n\}_{n=1}^{\infty}$ is a Cauchy sequence. Let
$y_n \to y_0\in M$ as $n \to \infty$. Then, since $x_n=\alpha_{n} x_1 +y_n$, we conclude that
\[
x_0=\mathop {\lim}\limits_{n \to \infty} x_n=\mathop {\lim}\limits_{n \to \infty} \left(\alpha_{n} x_1 +y_n\right)=\alpha_{0} x_1 +y_0\in M\left(x_1\right).
\]
It follows that $M(x_1)$ is a closed subspace of $X$. Moreover, we define the Banach space $\left(M(x_1), \|\cdot\|_1\right)$, where
\[
\|\alpha x_1 +y\|_1= \|\alpha x_1\| +\|y\|~\quad~~~~~~~~~\mathrm{whenever}~~~~~~~~~~~~\quad~~~\alpha \in \mathbb  R,~y\in M.
\]
Therefore, since $\left\| \alpha x_1 +y\right\| \leq \|\alpha x_1\| +\|y\|$, we obtain that $\|\cdot\|$ and $\|\cdot\|_1$ are two equivalent norm. Hence
there exists a real number $d_0\in (0,+\infty)$ such that
\[
 \left\| \alpha x_1 +y\right\| \leq \|\alpha x_1\| +\|y\|\leq d_0 \left\| \alpha x_1 +y\right\|~\quad~~~~~~~~~\mathrm{whenever}~~~~~~~~~~~~\quad~~~\alpha \in\mathbb  R,~y\in M.
\]
Pick a point ${x^*} \in {C^*}$. Define the linear functional $\overline{{x^*}}$ in the space $M(x_1)$, where
\[
\left\langle \overline{{x^*}} , \alpha x_1 +y \right\rangle= \left\langle {x^*} , y\right\rangle~\quad~~~~~~~~~\mathrm{whenever}~~~~~~~~\quad~~~~\alpha \in\mathbb  R,~y\in M.
\]
Noting that $C^*\subset B\left(M^{*}\right)$, we conclude that  $\sup \left\{\|{x^*}\|: {x^*}\in {C^*} \right\}\leq 1$.
Moreover, for every $y\in M$, we have the following inequalities
\[
\left|\left\langle \overline{{x^*}} , \alpha x_1 +y \right\rangle\right|  = \left|\left\langle x^* , y \right\rangle\right|  \leq \left\|x^*\right\| \left\|y\right\|\leq \left\|{x^*}\right\| \left(   \left\| \alpha x_1\right\|+ \left\|y\right\| \right)\leq 2\left(d_0\left\| \alpha x_1 +y \right\|\right).
\]
This implies that $\overline{{x^*}}$ is a bounded linear functional in $M\left(x_1\right)$. Therefore, by the Hahn-Banach theorem, there exists a functional $F_{x^*}\in X^{*}$ such that
\[
\|F_{x^*}\|= \|\overline{{x^*}}\|\leq 2d_0~\quad~~~~~~~~~\mathrm{and}~~~~~~~~~~~~\quad~~ F_{x^*}\left(z\right)= \langle \overline{{x^*}}, z\rangle~\quad~~~~~~~~~\mathrm{whenever}~~~~~~~~~~~~\quad~~z\in M\left(x_1\right).
\]
Hence, for every ${x^*} \in {C^*}$, there exists a functional $F_{x^*}\in X^{*}$ with $\|F_{x^*}\|\leq 2d_0$ such that $ F_{x^*}(x_1)=0$ and $ F_{x^*}(z)=x^*(z)$ whenever $z\in M$.
Define the weak$^{*}$ bounded closed convex subset $K_{2,2}^{*}$ of $ X^{*}$, where
\[
K_{2,2}^{*}=\overline{co}^{w^{*}}  \big\{  F_{x^*}\in X^{*}  :  x^*\in   C_2^*   \big\}.
\]
Then, by the definition of $K_{2,2}^{*}$, we get that $I^{*}(K_{2,2}^{*})= C_{2}^{*}$.
Moreover, we define the bounded  weak$^{*}$ closed convex subset $K_{2,1}^{*}$ of $X^{*}$, where
\[
K_{2,1}^{*}=K_{1}^{*}\cap (I^{*})^{-1}   \big(    \{ x^{*}\in C^{*}:x^{*}(y_2)\geq 0 \}   \big).
\]
Define the set $K_{2}^{*}= \overline{co}^{w^{*}}  (K_{2,1}^{*}\cup K_{2,2}^{*})$. Then $K_{2}^{*}$ is a bounded weak$^{*}$ closed convex subset of $X^{*}$.
We claim that $\sigma_{K_{2,1}^{*}}(x_1)>\sigma_{K_{2,2}^{*}}(x_1) $.
In fact, noting that  $\left\|x-x_1\right\| $ $<256^{-1}\beta_1$ and
$x=y_1$, by the formula (2.21), we obtain that
\[
\left\|    x_1-y_2         \right\|\leq\left\|    x_1-y_1         \right\|+\left\|    y_1-y_2        \right\|\leq \frac{1}{256}\beta_1+\frac{1}{4}\beta_1<\frac{1}{3}\beta_1.
\]
Since  $z^{*}\in C^{*}$ and $C^{*}\subset B\left(M^{*}\right)$, we obtain that $\left\|z^{*}\right\|\leq 1$.
Noting that $z^{*}(x)=a$, $I^{*}(x_{0}^{*})=z^{*}$  and $\beta_{1}= $ $\min\left\{ r/256, a/256\right\}$, by the inequalities $\left\|z^{*}\right\|\leq 1$ and $\left\|x_0^{*}\right\|$ $\leq 1$, we have the following inequalities
\begin{eqnarray*}
z^{*}\left(y_2\right) &=& z^{*}\left(x\right)-x_0^{*}\left(x-y_2\right)
\\
 &\geq& z^{*}\left(x\right)-x_0^{*}\left(x-x_1\right)-x_0^{*}\left(x_1-y_2\right)
\\
 &\geq& z^{*}\left(x\right)-\left\|x_0^{*}\right\|\left\|x-x_1\right\|-\left\|x_0^{*}\right\|\left\|x_1-y_2\right\|
\\
 &\geq& a-\frac{1}{256}\beta_1-\frac{1}{3}\beta_1 > 0.
\end{eqnarray*}
Then, by the definition of $C_{2}^{*}$, we have
$z^{*}\notin C_{2}^{*}$. Further, by the definition of $K_{2,1}^{*}$, we conclude that
$x_0^{*}\in K_{2,1}^{*}$. Moreover, by the definition of $K_{2,2}^{*}$, we conclude that $x^{*}\left(x_1\right)=0$ for every $x^{*}\in  K_{2,2}^{*}$.
From the proof of Step 1, we have $x_0^{*}\left(x_1\right)>a/2$. Therefore, since $x_0^{*}\in K_{2,1}^{*}\subset K_{1}^{*}$ and $x_0^{*}\left(x_1\right)>a/2$, we conclude that
\[
\sigma_{K_{2,1}^{*}}\left(x_1\right)\geq x_0^{*}\left(x_1\right)> \frac{1}{2}a>0=\sigma_{K_{2,2}^{*}}\left(x_1\right).
\]

\par Pick a point $x^{*}\in K_{1}^{*}$ such that $x^{*}(x_1)\geq a/2$. Then, since $\left\|    x_1-y_2         \right\|\leq \beta_1/3$, we conclude that $x^{*}(y_2)>0$. This implies that
$x^{*}\in K_{2,1}^{*} \subset K_{2}^{*} $. Hence $\sigma_{K_{1}^{*}}\left(x_1\right)$ $\leq\sigma_{K_{2}^{*}}\left(x_1\right)$.
Pick two points $x_1^{*}\in K_{2,1}^{*}$ and $x_2^{*}\in K_{2,2}^{*}$. Then, since $x_0^{*}\in K_{2,1}^{*}\subset K_{1}^{*}$, $ x_2^{*}\left(x_1\right)=0$ and $x_0^{*}\left(x_1\right)>a/2$, we obtain the following inequalities
\begin{eqnarray*}
\quad \quad\quad \quad \quad\quad\big\langle  \lambda x_1^{*}  +(1-\lambda ) x_2^{*} ,x_1        \big\rangle &=& \lambda x_1^{*}\left(x_1\right)+ (1-\lambda ) x_2^{*}\left(x_1\right)
\\
&\leq& \max \left\{\lambda x_1^{*}\left(x_1\right),0 \right\}
\\
&\leq& \max \left\{ x_1^{*}\left(x_1\right),0 \right\}
\\
\quad \quad &\leq& \sigma_{K_{2,1}^{*}}\left(x_1\right)\leq\sigma_{K_{1}^{*}}\left(x_1\right)~~~~~~~~~\quad\quad  \quad \quad\quad \quad~~~~~~~~~~~~~(2.22)
\end{eqnarray*}
for all $\lambda\in [0,1]$. Since $K_{2,1}^{*}$ and $K_{2,2}^{*}$ are weak$^{*}$ closed and convex, by the formula (2.22) and
$K_{2}^{*}= \overline{co}^{w^{*}}  \left(K_{2,1}^{*}\cup K_{2,2}^{*}\right)$, we conclude that $\sigma_{K_{1}^{*}}\left(x_1\right)\geq \sigma_{K_{2}^{*}}\left(x_1\right)$.
Hence $\sigma_{K_{1}^{*}}(x_1)= \sigma_{K_{2}^{*}}(x_1)$.
Pick $\eta\in \min \{   \sigma_{K_{2,1}^{*}}\left(x_1\right)-\sigma_{K_{2,2}^{*}}\left(x_1\right) ,~ \beta_1/256       \}$.
 We claim that
\[
I^{*}\left( K_{2}^{*}\right)\backslash \left[I^{*}\left( K_{2}^{*}\backslash S\left(x_1,\frac{1}{4}\eta, K_{2}^{*}\right)\right)\right]\neq \emptyset.~~~~~~~~~~~~~~~~~~~~~~~~~~~~~\eqno~~~~~~~~~~~~~~(2.23)
\]
In fact, we have shown that $K_{1}^{*}$ is a minimal element of $J$. Hence we obtain that
\[
I^{*}\left( K_{1}^{*}\right)\backslash \left[I^{*}\left( K_{1}^{*}\backslash S\left(x_1,\frac{1}{4}\eta, K_{1}^{*}\right)\right)\right]\neq \emptyset.
\]
We pick a point $u_0^{*}\in I^{*}\left( K_{1}^{*}\right)\backslash \left[I^{*}\left( K_{1}^{*}\backslash S\left(x_1,\eta/4, K_{1}^{*}\right)\right)\right]$. Then, by the separation theorem, there exists
a point $y\in S\left(M\right)$ and a real number $d\in (0,1)$ such that
\[
\sigma_{C^{*}}(y)-2d>u_0^{*}(y)-4d\geq \sup \left\{ y^*(y) :  y^*\in     I^{*}\left( K_{1}^{*}\backslash S\left(x_1,\frac{1}{4}\eta, K_{1}^{*}\right)\right)                 \right\}.
\]
Therefore, from the above inequalities and $I^{*}(K_{1}^{*})=C^*$, it is easy to see that
\[
S\left(y,C^*,d\right)\cap I^{*}\left( K_{1}^{*}\backslash S\left(x_1,\frac{1}{4}\eta, K_{1}^{*}\right)\right)=\emptyset.~~~~~~~~~~~~~~~~~~~~~~~~~~~~~~~~~~~~~~~~~~\eqno~~~~~~~~~~~~~~~~~~~~~~~~~~~~(2.24)
\]
Since the set $C^*$ is weak$^{*}$ compact, there exists a point $y^{*}\in S\left(y,C^*,d\right)$ such that $\sigma_{C^*}(y)=y^{*}(y)$.
Therefore, by the definition of $\eta$, we conclude that $S\left(x_1,\eta/4, K_{1}^{*}\right)$ $\subset K_{2,1}^{*}$. Moreover, since
$\sigma_{K_{1}^{*}}(x_1)= \sigma_{K_{2}^{*}}\left(x_1\right)$
and $K_{2}^{*}= \overline{co}^{w^{*}}  (K_{2,1}^{*}\cup K_{2,2}^{*})$, by the formula (2.24), it is easy to see that
\[
y^{*}\in I^{*}\left(  S\left(x_1,\frac{1}{4}\eta, K_{1}^{*}\right)\right) = I^{*}\left(  S\left(x_1,\frac{1}{4}\eta, K_{2,1}^{*}\right)\right) \subset I^{*}\left(  S\left(x_1,\frac{1}{4}\eta, K_{2}^{*}\right)\right).
\]
Suppose that there exists a point
$
x^{*}\in K_{2}^{*}\backslash S\left(x_1,\eta/4, K_{2}^{*}\right)
$
such that $I^{*}(x^{*})=y^{*}$.
Since $K_{2,1}^{*}$ and $K_{2,2}^{*}$ are two weak$^{*}$ closed convex subset of $X^{*}$, by the fact that
$K_{2}^{*}= \overline{co}^{w^{*}} (K_{2,1}^{*}\cup K_{2,2}^{*})$ and $x^{*}\in K_{2}^{*}$,
there exists a net
\[
\left\{ \lambda_{\alpha} x_{1,\alpha}^{*} + (1-\lambda_{\alpha} )x_{2,\alpha}^{*} \right\}_{\alpha\in\Delta} \subset  co \left(K_{2,1}^{*}\cup K_{2,2}^{*}\right)
\]
such that $\lambda_{\alpha} x_{1,\alpha}^{*} + (1-\lambda_{\alpha} )x_{2,\alpha}^{*} \xrightarrow{w^{*}} x^{*}$, where $x_{1,\alpha}^{*}\in K_{2,1}^{*}$, $x_{2,\alpha}^{*}\in K_{2,2}^{*}$
and $\lambda_{\alpha}\in [0,$ $1]$. Therefore, by the definition of $K_{2,2}^{*}$, we conclude that  $\langle I^{*}(x_{2,\alpha}^{*}),y_2 \rangle\leq 0$ for all $\alpha\in\Delta$.
Since $\{x^{*}\in M^{*}: \|x^{*}\|\leq a\}\subset \mathrm{int}\left({C^*}\right)$ and $y_2\in S(M)$, we conclude that $\sigma_{K_{1}^{*}}\left(y_2\right)\geq a$.
Pick a point $x_{\eta}^{*}\in S\left(x_1,\eta /4, K_{1}^{*}\right)$.
Then $\|x_{\eta}^{*}\|\leq 1$. Therefore,
since $x_0^{*}\left(x_1\right)>a/2 $ and $\left\|x_1-y_2\right\| \leq \beta_1/3$, we obtain the following inequalities
\[
~x_{\eta}^{*}\left(y_2\right)~=~x_{\eta}^{*}\left(x_1\right)- x_{\eta}^{*}\left(x_1-y_2\right) \quad\quad\quad \quad\quad\quad \quad\quad
\]
\[
\quad \quad\quad\quad \quad\quad\quad\quad \quad\quad\geq~ \sigma_{K_{1}^{*}}\left(x_1\right)~-~\frac{1}{4}\eta ~-~ x_{\eta}^{*}\left(x_1-y_2\right)\quad\quad\quad \quad\quad\quad \quad\quad\quad\quad\quad \quad\quad\quad \quad\quad\quad\quad\quad \quad\quad\quad \quad\quad\quad \quad\quad\quad\quad\quad \quad\quad\quad\quad\quad \quad\quad\quad \quad\quad
\]
\[
\quad \quad\quad\quad \quad\quad\quad\quad \quad\quad\geq ~\sigma_{K_{1}^{*}}\left(x_1\right)~-~\frac{1}{4}\eta~-~ \left\|x_{\eta}^{*} \right\| \cdot\left\|x_1-y_2\right\|\quad \quad\quad\quad \quad\quad\quad \quad\quad\quad \quad\quad\quad \quad\quad\quad \quad\quad\quad \quad\quad\quad \quad\quad
\]
\[
\quad \quad\quad\quad \quad\quad\quad\quad \quad\quad\geq ~x_0^{*}\left(x_1\right)~-~\frac{1}{4}\eta~-~ \left\|x_1-y_2\right\|\quad \quad\quad\quad \quad\quad\quad \quad\quad\quad \quad\quad\quad \quad\quad\quad \quad\quad\quad \quad\quad\quad \quad\quad
\]
\[
\quad \quad\quad\quad \quad\quad\quad\quad \quad\quad> ~\frac{1}{2}a~-~ \frac{1}{4}\eta~-~ \frac{1}{3}\beta_1>\frac{1}{4}a ~>~0. \quad \quad\quad\quad \quad\quad\quad \quad\quad\quad \quad\quad\quad \quad\quad\quad \quad\quad\quad \quad\quad\quad \quad\quad\quad\quad\quad \quad\quad\quad \quad\quad\quad\quad\quad \quad\quad
\]
Since $I^{*}( K_{1}^{*})=C^*$  and   $\langle I^{*}(x_{2,\alpha}^{*}),y_2 \rangle\leq 0$, from the above inequalities, we have
\[
I^{*}(x_{2,\alpha}^{*})\in I^{*}\left( K_{1}^{*}\backslash S\left(x_1,\frac{1}{4}\eta, K_{1}^{*}\right)\right).
\]
Therefore, by the formula (2.24), we have $\langle I^{*}(x_{2,\alpha}^{*}),y \rangle\leq \sigma_{C^*}(y)-d $ for all $\alpha\in\Delta$.
Moreover, since $x_{1,\alpha}^{*}\in K_{2,1}^{*}$, we conclude that $\langle I^{*}(x_{1,\alpha}^{*}),y\rangle\leq \sigma_{C^*}(y)$ for all $\alpha\in\Delta$. Then, since $\lambda_{\alpha} x_{1,\alpha}^{*} + (1-\lambda_{\alpha} )x_{2,\alpha}^{*} \xrightarrow{w^{*}} x^{*}$ and $\langle I^{*}(x_{2,\alpha}^{*}),y \rangle\leq \sigma_{C^*}(y)-d $, we have
\begin{eqnarray*}
\sigma_{C^*}(y)&=& \left\langle I^{*}\left(x^{*}\right),y  \right\rangle
\\
 &=& \mathop {\lim}\limits_{\alpha\in\Delta}\left\langle I^{*}\left(\lambda_{\alpha} x_{1,\alpha}^{*} + (1-\lambda_{\alpha} )x_{2,\alpha}^{*}\right),y  \right\rangle
\\
&\leq & \mathop {\lim\sup}\limits_{\alpha\in\Delta} \left[ \lambda_{\alpha}\left\langle I^{*}(x_{1,\alpha}^{*}),y  \right\rangle \right] + \mathop {\lim\sup}\limits_{\alpha\in\Delta}\left[\left(1-\lambda_{\alpha} \right)\left\langle I^{*}(x_{2,\alpha}^{*}),y  \right\rangle \right]
\\
&\leq & \mathop {\lim\sup}\limits_{\alpha\in\Delta} \left[\lambda_{\alpha} \sigma_{C^*}(y) \right]+ \mathop {\lim\sup}\limits_{\alpha\in\Delta} \left[ \left(1-\lambda_{\alpha} \right) \left( \sigma_{C^*}(y)-d \right) \right]
\\
&\leq&  \sigma_{C^*}(y).
\end{eqnarray*}
It follows that $\lambda_{\alpha}\to 1$. Therefore, since $x_{1,\alpha}^{*}\in K_{2,1}^{*}$, we conclude that
$x^{*}\in  K_{2,1}^{*} \subset K_{1}^{*}$.  Since $ I^{*}(x^{*})=y^{*}$, $y^{*}\in S\left(y,C^*,d\right)$ and $x^{*}\in  K_{1}^{*}$, by
the formula (2.24), we have $x^{*}\in  S\left(x_1,\eta/4, K_{1}^{*}\right)$. Since $\eta\leq \sigma_{K_{2,1}^{*}}(x_1)-\sigma_{K_{2,2}^{*}}(x_1)$ and $\sigma_{K_{1}^{*}}\left(x_1\right)= \sigma_{K_{2}^{*}}\left(x_1\right)$, by the fact that $S\left(x_1,\eta/4, K_{1}^{*}\right)\subset K_{2,1}^{*}$
and $K_{2}^{*}= \overline{co}^{w^{*}} \left(K_{2,1}^{*}\cup K_{2,2}^{*}\right)$,   we  have
\[
x^{*}\in  S\left(x_1,\frac{1}{4}\eta, K_{1}^{*}\right)= S\left(x_1,\frac{1}{4}\eta, K_{2,1}^{*}\right) \subset  S\left(x_1,\frac{1}{4}\eta, K_{2}^{*}\right),
\]
which contradicts $x^{*}\in K_{2}^{*}\backslash S\left(x_1,\eta/4, K_{2}^{*}\right)$. Hence the formula (2.23) holds.
Let
$H_2^{*}= K_{2,2}^{*} \cap (I^{*})^{-1}(N_2^{*})$. Then, since $0\in \mathrm{int}\left(C^{*}\right)$ and the definition of $K_{2,2}^{*}$, we obtain that
$H_2^{*}$ is a nonempty bounded weak$^{*}$ closed convex subset of $X^{*}$.
Define
\[
X_2^{*}= \left\{ \lambda x^{*}  :  \lambda\in \mathbb R^{+},~ x^{*}\in  K_2^{*} ,~~ x^{*}|_{M}\notin \mathrm{int}|_{M^{*}}(C^{*} )              \right\}.
\]
Similar to the proof of Step 1, we have $K_{2}^{*}\supset X_2^{*} \cap (I^{*})^{-1}(C^{*})$. Define the set
\[
M_2^{*}= \left\{ \lambda x^{*}  :  \lambda\in \mathbb R^{+},~ x^{*}\in  H_2^{*} ,~~ x^{*}|_{M}\notin \mathrm{int}|_{M^{*}}(C^{*} )              \right\}.
\]
Then, by the definition of $M_2^{*}$ and $H_2^{*}= K_{2,2}^{*} \cap (I^{*})^{-1}(N_2^{*})$,  we get that $M_2^{*}\subset X_2^{*}$.
Since $0\in \mathrm{int}\left(C^{*}\right)$, we define the nonempty subset $L_{2}^{*}$ of $X^{*}$, where
\[
L_{2}^{*}=M_{2}^{*}\cap (I^{*})^{-1} (N_{2}^{*})\subset X_2^{*}.
\]
Then, by the definitions of $K_{2,2}^{*}$ and $M_{2}^{*}$, we obtain that $I^{*}(L_{2}^{*})=N_{2}^{*}$.
Since $N_{2}^{*}$ is a bounded weak$^{*}$ closed convex subset of $M^{*}$ and $K_{2}^{*}\supset X_2^{*} \cap (I^{*})^{-1}(C^{*})$, by the definition of $L_{2}^{*}$, we obtain that $L_{2}^{*}$ is a bounded subset of $X^{*}$.
Define
\[
D_{2}^{*}=\overline{co}^{w^{*}} \left(K_{2}^{*}\cup L_{2}^{*}\right)~~\quad~~~~~~~~~~~~~~\mathrm{and}~~~~~~~~~~~~~~~~~~\quad~~~~~~~~D_{2,0}^{*}=\overline{co}^{w^{*}} \left(C^{*}\cup N_{2}^{*}\right).
\]
Then $D_{2}^{*}$ is a bounded weak$^{*}$ closed convex subset of $X^{*}$ and $D_{2,0}^{*}$ is a bounded weak$^{*}$ closed convex subset of $M^{*}$. Moreover, we have $D_{2,0}^{*}=I^{*}(D_{2}^{*})$.

\par Pick a real number $\eta_{2}\in (0,\min\left\{\eta_{1}/4, \eta\right\})$. Pick $x^{*}\in K_{1}^{*}$ such that $x^{*}(x_1)>a/2$. Then, since
$\|x_1-y_2\|<\beta_1/3$, we obtain that $x^{*}(y_2)>0$. This implies that $x^{*}\in K_{2,1}^{*}$.
Therefore, since $K_{2,1}^{*} \subset K_{1}^{*}$, it is easy to see that
$\sigma_{K_{2,1}^{*}}(x_1)=\sigma_{K_{1}^{*}}(x_1)$.
Moreover, since $\sigma_{K_{2,1}^{*}}(x_1)=\sigma_{K_{1}^{*}}(x_1)$ and $K_{2,1}^{*} \subset K_{1}^{*}$, we conclude that
\[
S\left(x_1,K_{2,1}^{*},\frac{1}{4}\eta_2 \right)+ B\left(0,\frac{1}{4}\eta_2 \right)\subset S\left(x_1,K_{1}^{*},\frac{1}{4}\eta_2 \right)+ B\left(0,\frac{1}{4}\eta_2 \right).
\]
Noting that $K_{2,1}^{*}$ and $K_{2,2}^{*}$ are two bounded weak$^{*}$ closed convex sets  and $\sigma_{K_{2,1}^{*}}(x_1)$ $>\sigma_{K_{2,2}^{*}}(x_1) $,
by Lemma 2.3,  there exists a real number $\theta_2\in (0,\eta_{2})$ such that
\begin{eqnarray*}
\quad\quad\quad\quad S\left(x_1,K_{2}^{*},\frac{1}{4}\theta_2 \right)&\subset& S\left(x_1,K_{2,1}^{*},\frac{1}{4}\eta_2 \right)+ B\left(0,\frac{1}{4}\eta_2 \right)
\\
&\subset& S\left(x_1,K_{1}^{*},\frac{1}{4}\eta_2 \right)+ B\left(0,\frac{1}{4}\eta_2 \right).\quad\quad\quad\quad\quad (2.25)
\end{eqnarray*}
Moreover, by the definition of $L_{2}^{*}$, we have $x^{*}(x_1)=0$ for all $x^{*}\in L_{2}^{*}$. Then, since
$x_0^{*}\left(x_1\right)>a/2$ and $x_0^{*}\in K_{2,1}^{*}$, we conclude that $\sigma_{K_{2}^{*}}(x_1)>\sigma_{L_{2}^{*}}(x_1) $.
Since $L_{2}^{*}$ is a bounded set,  by  Lemma 2.3, there exists $\lambda_2\in (0,\min\{\theta_{2},\eta/4\} )$ such that
\[
S\left(x_1,D_{2}^{*},\lambda_2\right )\subset S\left(x_1,K_{2}^{*},\frac{1}{4}\theta_{2} \right)+B\left(0,\frac{1}{4}\theta_{2} \right).~~~~\eqno~~~~~~~~~(2.26)
\]
Therefore, by the formula (2.25) and $\theta_2\in (0,\eta_{2})$, we have the following formulas
\[
S\left(x_1,D_{2}^{*},\lambda_2\right )- S\left(x_1,D_{2}^{*},\lambda_2\right )\quad\quad \quad \quad \quad \quad \quad \quad \quad \quad \quad \quad \quad \quad\quad \quad \quad\quad \quad
\]
\[
\subset S\left(x_1,K_{2}^{*},\frac{1}{4}\theta_{2} \right)~-~S\left(x_1,K_{2}^{*},\frac{1}{4}\theta_{2} \right)~+~B\left(0,\frac{1}{4}\theta_{2} \right)~-~B\left(0,\frac{1}{4}\theta_{2} \right)\quad \quad \quad \quad \quad \quad \quad \quad \quad \quad\quad \quad \quad \quad \quad \quad \quad \quad \quad \quad \quad \quad \quad \quad \quad \quad\quad \quad \quad \quad \quad \quad
\]
\[
\subset S\left(x_1,K_{2}^{*},\frac{1}{4}\theta_{2} \right)~-~S\left(x_1,K_{2}^{*},\frac{1}{4}\theta_{2} \right)~+~B\left(0,\frac{1}{2}\theta_{2} \right)\quad \quad \quad \quad \quad \quad \quad \quad \quad \quad \quad \quad \quad \quad \quad \quad \quad \quad\quad \quad \quad \quad \quad \quad \quad \quad
\]
\[
\subset S\left(x_1,K_{1}^{*},\frac{1}{4}\eta_2 \right)~-~ S\left(x_1,K_{1}^{*},\frac{1}{4}\eta_2 \right)~+~ B\left(0,\frac{1}{4}\eta_2 \right)~+~ B\left(0,\frac{1}{4}\eta_2 \right)~+~B\left(0,\frac{1}{2}\theta_{2} \right)\quad \quad\quad \quad \quad \quad \quad \quad \quad \quad \quad \quad \quad \quad \quad \quad \quad \quad
\]
\[
\subset S\left(x_1,K_{1}^{*},\frac{1}{4}\eta_2 \right)~-~ S\left(x_1,K_{1}^{*},\frac{1}{4}\eta_2 \right)~+~ B\left(0,\eta_2 \right).\quad \quad \quad \quad \quad \quad \quad \quad\quad \quad \quad  (2.27)\quad \quad \quad \quad  \quad \quad \quad \quad \quad \quad \quad \quad
\]
We next prove that $ D_{2,0}^{*}\backslash I^{*}(D_{2}^{*} \backslash S\left(x_1,\lambda_2,D_{2}^{*}\right )) \neq \emptyset$. In fact, we have proved that
\[
I^{*}\left( K_{2}^{*}\right)\backslash I^{*}\left( K_{2}^{*}\backslash S\left(x_1,\lambda_2, K_{2}^{*}\right)\right)\neq \emptyset.
\]
Pick a point $y^{*}\in I^{*}( K_{2}^{*})\backslash I^{*}\left( K_{2}^{*}\backslash S(x_1,\lambda_2, K_{2}^{*})\right)$. Then
$y^{*}\notin  I^{*}\left( K_{2}^{*}\backslash S(x_1,\lambda_2, K_{2}^{*})\right)$.
Therefore, since $I^{*}( K_{2}^{*})=C^{*}$,
there exists a point $x^{*}\in S\left(x_1,\lambda_2, K_{2}^{*}\right)$ such that
\[
y^{*}=I^{*}(x^{*})\notin I^{*}\left( K_{2}^{*}\backslash S\left(x_1,\lambda_2, K_{2}^{*}\right)\right).
\]
Suppose that $I^{*}(x^{*})\in I^{*}\left(D_{2}^{*} \backslash S\left(x_1,\lambda_2,D_{2}^{*} \right)\right)$. Then there exists a point $z^*\in D_{2}^{*} \backslash $ $ S\left(x_1,\lambda_2,D_{2}^{*}\right )$
such that $I^{*}(z^{*})=I^{*}(x^{*})$. Noting that $x^{*}\in S\left(x_1,\lambda_2, K_{2}^{*}\right)$, we have
\[
I^{*}(z^{*})=I^{*}(x^{*})\in I^{*}(K_{2}^{*})= C^{*}.
\]
Therefore, by $z^*\in D_{2}^{*} \backslash S(x_1,\lambda_2,D_{2}^{*})$ and $I^{*}(z^{*})\notin I^{*}\left( K_{2}^{*}\backslash S\left(x_1,\lambda_2, K_{2}^{*}\right)\right)$, we have
$z^*\notin K_{2}^{*}$. Moreover, by the definition of $D_{2}^{*}$, we obtain that
$z^*\in X_{2}^{*}$.
Noting that $K_{2}^{*}\supset X_{2}^{*} \cap (I^{*})^{-1} C^{*}$,
by $z^*\notin K_{2}^{*}$ and $z^*\in X_{2}^{*}$, we conclude that $I^{*}(z^{*})\notin C^{*}$, a contradiction.
 Hence we get that
$I^{*}(x^{*})\notin $ $I^{*}(D_{2}^{*} \backslash S\left(x_1,\lambda_2,D_{2}^{*}\right ))$. It follows that
\[
I^{*}\left(x^{*}\right)\in D_{2,0}^{*}\backslash I^{*}\left(D_{2}^{*} \backslash S\left(x_1,\lambda_2,D_{2}^{*}\right )\right).
\]
Hence we get that $ D_{2,0}^{*}\backslash I^{*}\left(D_{2}^{*} \backslash S(x_1,\lambda_2,D_{2}^{*} )\right) \neq \emptyset$. Similar to the   proof of Step 1, from
$ D_{2,0}^{*}\backslash I^{*}\left(D_{2}^{*} \backslash S(x_1,\lambda_2,D_{2}^{*} )\right) \neq \emptyset$, we conclude that there exists a point $y_3\in $ $S(M)$ and a real number
$\beta_3\in (0, \min \{\beta_2/256, \lambda_2/256\})$ such that
\[
S\left({y_3},D_{2,0}^{*},\beta_{3}\right)\subset D_{2,0}^{*}\backslash I^{*}\left(D_{2}^{*} \backslash S\left(x_1,\lambda_2,D_{2}^{*}\right )\right).
\]
Then, since $I^{*}\left(D_{2}^{*}\right)=D_{2,0}^{*}$, we conclude that $S\left({y_3},D_{2,0}^{*},\beta_{3}\right)\subset I^{*}\left( S\left(x_1,\lambda_2,D_{2}^{*} \right)\right)$.
We pick a point $y^{*}\in S\left({y_3},D_{2,0}^{*},\beta_{3}\right)$. Since $\eta_{2}\in (0,\eta_{1}/4)$, by the formulas (2.25) and (2.26),
there exists a point
$x^{*} \in S\left(x_1,\lambda_2,D_{2}^{*} \right)$ with $I^{*}(x^{*})=y^{*}$ such
that $\|x^{*}\|\leq 2$.
Moreover, from the proof of Step 1, there exists a point $x_0^{*}\in K_{1}^{*}$ such that $I^{*}(x_0^{*})=z^{*}$.
Noting that $\|x-x_1\|<256^{-1}\beta_1(\gamma_0+1)^{-1}(\theta+1)^{-1}$,  $I^{*}(x_0^{*})=z^{*}$ and  $\lambda_2\in (0,\min\{\theta_{2},\eta\})$, from $z^{*}(x)=a$ and $\|y_2-x_1\|<\beta_1/3$, we obtain that
\[
y^{*} (y_2)=x^{*} (y_2)~=~x^{*} (x_1)-x^{*} (x_1-y_2)\quad \quad\quad \quad\quad \quad\quad \quad
\]
\[
  \quad\quad\quad \quad \quad \quad\quad \quad\quad \quad\quad \quad\geq~ \sigma_{D_{2}^{*}}\left(x_1\right)~-~\lambda_{2}~-~x^{*} \left(x_1-y_2\right)\quad \quad \quad \quad\quad \quad \quad \quad\quad \quad \quad \quad\quad \quad \quad \quad\quad \quad \quad \quad\quad \quad \quad \quad\quad\quad \quad \quad \quad
\]
\[
 \quad \quad\quad \quad\quad \quad\quad \quad \quad \quad\quad \quad\geq~ x_0^{*}\left(x_1\right)~-~x^{*} \left(x_1-y_2\right)~-~\lambda_{2}\quad \quad \quad \quad\quad \quad \quad \quad\quad \quad \quad \quad\quad \quad \quad \quad\quad \quad \quad \quad\quad \quad \quad \quad\quad\quad \quad \quad \quad
\]
\[
  \quad\quad\quad \quad \quad\quad \quad\quad \quad \quad\quad \quad\geq~ x_0^{*}\left(x\right)~- ~ x_0^{*} \left(x-x_1\right)       ~-~ x^{*} \left(x_1-y_2\right)~-~\lambda_{2}\quad \quad \quad \quad\quad \quad \quad \quad\quad \quad \quad \quad\quad \quad \quad \quad\quad \quad \quad \quad\quad \quad \quad \quad\quad\quad \quad \quad \quad
\]
\[
  \quad\quad\quad\quad \quad\quad \quad \quad \quad \quad\quad \quad\geq~ z^{*}\left(x\right)~-~  \left\|x_0^{*}\right\| \left\|x_1-x\right\|     ~ -~\left\|x^{*} \right\|\left\|x_1-y_2\right\|~-~\lambda_{2}\quad \quad \quad \quad\quad \quad \quad \quad\quad \quad \quad \quad\quad \quad \quad \quad\quad \quad \quad \quad\quad \quad \quad \quad
\]
\[
 \quad \quad\quad \quad \quad\quad \quad\quad \quad \quad\quad \quad\geq~ a~-~\frac{1}{256}\beta_1~-~\frac{2}{3}\beta_1~-~\lambda_{2}>\frac{1}{2}a. \quad \quad \quad \quad\quad \quad \quad \quad\quad \quad \quad \quad\quad \quad \quad \quad\quad \quad \quad \quad\quad \quad \quad \quad
\]
Since $y^{*}\in S\left({y_3},D_{2,0}^{*},\beta_{3}\right)$ is arbitrary, by the above inequalities, we obtain that
\[
S\left({y_3},D_{2,0}^{*},\beta_{3}\right)\subset \left\{ x^{*}\in M^{*} :    x^{*}(y_2)>0                 \right\}.
\]
Since $S\left({y_3},D_{2,0}^{*},\beta_{3}\right)\subset I^{*}\left( S\left(x_1,\lambda_2,D_{2}^{*}\right )\right)$ and $S\left(x_1,\lambda_2,D_{2}^{*} \right)\subset S\left(x_1,D_{2}^{*} ,\lambda_2\right)$, by the formula (2.27), we obtain the following formula
\[
S\left({y_3},D_{2,0}^{*},\beta_{3}\right)- S\left({y_3},D_{2,0}^{*},\beta_{3}\right)  \subset        I^{*}\left( S\left(x_1,\lambda_2,D_{2}^{*} \right)- S\left(x_1,\lambda_2,D_{2}^{*} \right)   \right)        \subset I^{*}(V_2),
\]
where
\[
V_2= \left[ S\left(x_1,K_{1}^{*},\frac{1}{4}\eta_2 \right)- S\left(x_1,K_{1}^{*},\frac{1}{4}\eta_2 \right)+ B\left(0,\eta_2 \right)\right].
\]
Noting that $D_{2,0}^{*}=\overline{co}^{w^{*}} \left(C^{*}\cup N_{2}^{*}\right)$ and $S\left({y_3},D_{2,0}^{*},\beta_{3}\right)\subset \{ x^{*}\in M^{*} :    x^{*}(y_2)>0                 \}$,
by the fact that $z^{*}(y_2)>0$ and Lemma 2.4, we conclude that
\[
S\left({y_3},C^*,\beta_{3}\right)-S\left({y_3},C^*,\beta_{3}\right)  \subset        I^{*}\left( S\left(x_1,\lambda_2,D_{2}^{*} \right)- S\left(x_1,\lambda_2,D_{2}^{*}\right )   \right)   \subset I^{*}\left(V_2\right),
\]
\[
S\left({y_3},C^*,\beta_{3}\right)\subset \{ {x^*} \in {M^*}:{x^*}(y_2) > 0\}~~\quad~~~~~~\mathrm{and}~~~~~~~~~\quad~~~~~~~~~\|y_2-y_3\|<\frac{1}{2^{2}}\beta _2.
\]
Similar to the  proof of Step 1, we conclude that $S\left(y_3,{C^*},\beta_{3}\right)\subset  S\left(y_2,{C^*},\beta_{2}\right)$.

\par \textbf{Step 3.} In this step, we perform the iteration via the second-stage scheme and calibrate the relevant parameters to sustain iterative convergence. This procedure is necessary for infinite iteration processes.

\par Repeat the previous process, for every natural number $n\in N$,
we  define the weak$^{*}$ bounded closed convex subset $N_{n}^{*}$ of $M^*$, where
\[
N_{n}^{*}=\left\{ {x^*} \in M^*:\left\| {{x^*}} \right\| \leq \left(\frac{64}{\beta_{n}\cdot 2^{-n}}\right)  \sup_{{x^*} \in {C^*}} \left\| {{x^*}} \right\|,~{x^*}(y_n) = 0\right\} .
\]
Define the weak$^{*}$ closed convex set $C_{n}^{*}=\{x^{*}\in C^{*}:x^{*}(y_n)\leq 0 \}$. Then we obtain that $C_{n}^{*}$ is a bounded weak$^{*}$
closed convex set. Repeat the previous proof, there exist $\{y_i\}_{i=1}^{n-1}\subset S(M)$ and $\{\beta_i\}_{i=1}^{n-1}\subset \mathbb N$ such that
\[
\big\|y_i-y_{i+1}\big\|<        \frac{1}{2^{i}}\beta_{i}~~~~~~\quad~~~~~~~  \mathrm{and }~~~~~~~\quad~~~~~~~~~~0<\beta_{i+1}< \frac{1}{256}\beta_{i}~~~~~~~~\quad~~~~~~ \mathrm{for}~~~\mathrm{ all}~~~~~~~~~~~\quad~~~~~ 1\leq i \leq n-1.
\]
Since $\|x-x_1\|<256^{-1}\beta_1$,
by the triangle inequality and the above inequalities, we obtain the following inequalities
\[
\big\|    x_1-y_n         \big\| ~=~ \left\|    ( x_1-y_1 )       + \sum_{i=1}^{n-1}\left(y_i-y_{i+1}\right)  \right\|
\]
\[
\quad \quad \quad \quad \quad \quad \quad \quad \quad \quad \quad \quad \quad \leq~ \big\|    x_1-y_1         \big\|~+~ \left\| \sum_{i=1}^{n-1}\left(y_i-y_{i+1}\right)  \right\|\quad \quad \quad \quad \quad \quad \quad \quad \quad \quad \quad \quad \quad \quad \quad \quad \quad \quad \quad \quad \quad \quad \quad \quad
\]
\[
\quad \quad \quad \quad \quad \quad \quad \quad \quad \quad \quad \quad \quad\leq~ \big\|    x_1-y_1         \big\|~+~ \left(\sum_{i=1}^{n-1} \big\| y_i-y_{i+1}   \big\| \right)\quad \quad \quad \quad \quad \quad \quad \quad \quad \quad \quad \quad \quad \quad \quad \quad \quad \quad \quad \quad \quad \quad \quad \quad \quad \quad \quad \quad \quad \quad \quad \quad
\]
\[
\quad \quad \quad \quad \quad \quad \quad \quad \quad \quad \quad \quad \quad\leq~ \frac{1}{256}\beta_1~+~\frac{2}{3}\beta_1~<~\frac{3}{4}\beta_1. \quad \quad \quad \quad \quad \quad \quad \quad \quad \quad \quad \quad \quad \quad \quad \quad \quad \quad \quad \quad \quad \quad \quad \quad \quad \quad \quad \quad \quad \quad \quad \quad\quad \quad \quad \quad \quad \quad \quad \quad \quad \quad \quad \quad \quad
\]
Therefore, by an argument similar to that in Step 2,
there exists a bounded weak$^{*}$ closed convex subset $K_{n,2}^{*}$ of $X^{*}$ with $I^{*}(K_{n,2}^{*})= C_{n}^{*}$ such that
\[
\big\langle y^{*}, x_1\big\rangle=0~~~\quad~~~~~~~~~\mathrm{whenever}~~~~~~~~~~~~\quad~~y^{*}\in K_{n,2}^{*}.
\]
Moreover, we define the bounded weak$^{*}$ closed convex subset $K_{n,1}^{*}$ of $X^{*}$, where
\[
K_{n,1}^{*}=K_{1}^{*}\cap (I^{*})^{-1}   \big(    \{ x^{*}\in C^{*}:x^{*}(y_n)\geq 0 \}   \big).
\]
Define the bounded weak$^{*}$ closed convex subset $K_{n}^{*}= \overline{co }^{w^{*}}(K_{n,1}^{*}\cup K_{n,2}^{*})$ of $X^{*}$.

\par We claim that $ \sigma_{K_{n,1}^{*}}\left(x_1\right)>\sigma_{K_{n,2}^{*}}\left(x_1\right)$.
In fact, from the proof of Step 1, there exists $x_0^{*}\in  K_{1}^{*}$ such that $I^{*}\left(x_0^{*}\right)=z^{*}$. Since $\|x-x_1\|<256^{-1}\beta_1(\gamma_0+1)^{-1}(\theta+1)^{-1}$,
$\beta_{1}=\min\left\{ r/256,a/256\right\}$ and
$\left\|    x_1-y_n         \right\| \leq \left({3}/{4}\right)\beta_1$, from
$z^{*}(x)=a$ and $\|z^{*}\|\leq 1$, we obtain the following inequalities
\begin{eqnarray*}
z^{*}(y_n) &\geq& z^{*}(x)-z^{*}(x-x_1)-z^{*}(x_1-y_n)
\\
 &\geq& a-\left\|z^{*}\right\|\left\|x-x_1\right\|-\left\|z^{*}\right\|\left\|x_1-y_n\right\|
\\
 &\geq& a-\frac{1}{256}\beta_1-\frac{3}{4}\beta_1 > 0.
\end{eqnarray*}
Therefore, by the definition of $C_{n}^{*}$, we  have
$z^{*}\notin C_{n}^{*}$. Therefore, since $I^{*}(x_0^{*})=z^{*}$
and the definition of $ K_{n,1}^{*}$, we conclude that
$x_0^{*}\in K_{n,1}^{*}$. From the proof of Step 2, we conclude that $x_0^{*}\left(x_1\right)>a/2$.
Moreover, by the definition of $  K_{n,2}^{*}$, we conclude that $x^{*}(x_1)=0$ for every $x^{*}\in  K_{n,2}^{*}$.
Therefore, since $x_0^{*}\left(x_1\right)>a/2$, we conclude that $\sigma_{K_{n,1}^{*}}\left(x_1\right)>\sigma_{K_{n,2}^{*}}\left(x_1\right) $. Moreover, by the definition of $K_{n}^{*}$ and $x_0^{*}\in K_{n,1}^{*}$,
we obtain that
$\sigma_{K_{1}^{*}}(x_1)=\sigma_{K_{n}^{*}}(x_1)$
and $\sigma_{K_{n}^{*}}(x_1)>a/2$.
Similar to the proof of Step 2, we conclude that
\[
I^{*}\left( K_{n}^{*}\right)\backslash I^{*}\left( K_{n}^{*}\backslash S\left(x_1,\frac{1}{4}\overline{\eta}, K_{n}^{*}\right)\right)\neq \emptyset~~~~~~~~\eqno~~~~~~~~~~~~~~(2.28)
\]
for sufficiently small $\overline{\eta}\in \mathbb R^{+}$.  Define the weak$^{*}$ closed subspace $X_n^{*}$ of $X^{*}$, where
\[
X_n^{*}= \left\{ \lambda x^{*}  :  \lambda\in \mathbb R^{+},~ x^{*}\in  K_n^{*},~ x^{*}|_{M}\notin \mathrm{int}|_{M^{*}} (C^{*})                \right\}.
\]
Then, similar to the proof of Step 1, we obtain that $K_{n}^{*}\supset X_n^{*} \cap (I^{*})^{-1}(C^{*})$. Let $H_n^{*}= K_{n,2}^{*} \cap (I^{*})^{-1}(N_n^{*})$.
Then, since $ 0\in \mathrm{int}\left(C^{*}\right)$ and the definition of $K_{n,2}^{*}$, we obtain that
$H_n^{*}$ is a nonempty bounded weak$^{*}$ closed convex subset of $X^*$. Define
\[
M_n^{*}= \left\{ \lambda x^{*}  :  \lambda\in\mathbb R^{+},~ x^{*}\in  H_n^{*} ,~~ x^{*}|_{M}\notin \mathrm{int}|_{M^{*}}(C^{*} )              \right\}.
\]
Then $M_n^{*}\subset X_n^{*}$. Since $0\in \mathrm{int}C^{*}$, we define the nonempty subset $L_{n}^{*}$ of $X^{*}$, where
\[
L_{n}^{*}=M_{n}^{*}\cap (I^{*})^{-1} \left(N_{n}^{*}\right)\subset X_n^{*}.
\]
Then, by the definition of $M_n^{*}$,  we get that $I^{*}(L_{n}^{*})=N_{n}^{*}$.
Since $N_{n}^{*}$ is a bounded subset of $M^{*}$ and $K_{n}^{*}\supset  X_n^{*} \cap (I^{*})^{-1}(C^{*})$, by the definition of $L_{n}^{*}$,
we obtain that $L_{n}^{*}$ is a bounded subset of $X^{*}$.
Define the two sets
\[
D_{n}^{*}=\overline{co}^{w^{*}} \left(K_{n}^{*}\cup L_{n}^{*}\right)~~\quad~~~~~~~~~~~~~~\mathrm{and}~~~~~~~~~~~~~~~~~~\quad~~~~~~~~D_{n,0}^{*}=\overline{co}^{w^{*}} \left(C^{*}\cup N_{n}^{*}\right).
\]
Then $D_{n}^{*}$ is a bounded weak$^{*}$ closed convex subset of $X^{*}$ and $D_{n,0}^{*}$ is a bounded weak$^{*}$ closed convex subset of $M^{*}$.
Moreover, we have $D_{n,0}^{*}=I^{*}(D_{n}^{*})$.

\par We have defined $\eta_{1}$ and $\eta_{2}$. Therefore,
by the preceding proof, $\eta_{n-1}$ has already been defined when the iteration proceeds to $n-1$.
Pick $\eta_{n}\in \left(0,\min\left\{ \eta_{n-1}/4, \overline{\eta}\right\}\right)$.
Pick $x^{*}\in K_{1}^{*}$ such that $x^{*}(x_1)>a/2$. Then, since
$\|x_1-y_n\|\leq (3/4) \beta_1$, we have $x^{*}(y_n)>0$. Therefore, by the definition of $K_{n,1}^{*}$, we get that $x^{*}\in K_{n,1}^{*}$.
It follows that
$\sigma_{K_{n,1}^{*}}(x_1)=\sigma_{K_{1}^{*}}(x_1)$.
Moreover, since $K_{n,1}^{*}$ and $K_{n,2}^{*}$ are two bounded weak$^{*}$ closed convex subsets of $X^{*}$,
$\sigma_{K_{n,1}^{*}}\left(x_1\right)=\sigma_{K_{1}^{*}}\left(x_1\right) $ and $K_{n}^{*}= \overline{co }^{w^{*}}\left(K_{n,1}^{*}\cup K_{n,2}^{*}\right)$, by the inequality
$\sigma_{K_{n,1}^{*}}(x_1)>\sigma_{K_{n,2}^{*}}(x_1) $ and Lemma 2.3, there exists a real number $\theta_n \in (0,\eta_{n})$ such that
\begin{eqnarray*}
\quad\quad\quad\quad S\left(x_1,K_{n}^{*},\frac{1}{4}\theta_n \right)&\subset& S\left(x_1,K_{n,1}^{*},\frac{1}{4}\eta_n \right)+ B\left(0,\frac{1}{4}\eta_n \right)
\\
&\subset& S\left(x_1,K_{1}^{*},\frac{1}{4}\eta_n \right)+ B\left(0,\frac{1}{4}\eta_n \right).\quad\quad\quad\quad\quad (2.29)
\end{eqnarray*}
Moreover, by the definition of $L_{n}^{*}$, we get that $x^{*}(x_1)=0$ for every $x^{*}\in L_{n}^{*}$.
Then, since $\sigma_{K_{n}^{*}}(x_1)>a/2$, we have $\sigma_{K_{n}^{*}}(x_1)>\sigma_{L_{n}^{*}}(x_1) $.
Since $K_{n}^{*}$ is a bounded weak$^{*}$ closed convex subset of $X^{*}$,
$L_{n}^{*}$ is a bounded subset $X^{*}$  and
$\sigma_{K_{n}^{*}}(x_1)>\sigma_{L_{n}^{*}}(x_1) $, by  Lemma 2.3,
there exists a real number $\lambda_n\in (0,\min\{\theta_{n},\overline{\eta}/4\} )$ such that
\[
S\left(x_1,D_{n}^{*},\lambda_n\right )\subset S\left(x_1,K_{n}^{*},\frac{1}{4}\theta_{n} \right)+B\left(0,\frac{1}{4}\theta_{n} \right).~~~~~~~\eqno~~~~~~~~~~~~~~(2.30)
\]
Moreover, noting that $\lambda_n\in \left(0,\min\left\{\theta_{n},\overline{\eta}/4\right\} \right)$ and $\theta_n\in (0,\eta_{n})$, by the formulas (2.29) and (2.30), we  obtain the following inclusion relation
\begin{eqnarray*}
\quad \quad \quad \quad &&S\left(x_1,D_{n}^{*},\lambda_n\right )- S\left(x_1,D_{n}^{*},\lambda_n\right )
\\
&\subset& \left[ S\left(x_1,K_{1}^{*},\frac{1}{4}\eta_n \right)- S\left(x_1,K_{1}^{*},\frac{1}{4}\eta_n \right)+ B\left(0,\eta_n \right)\right].\quad \quad \quad \quad (2.31)
\end{eqnarray*}
Similar to the proof of Step 2, we conclude that $ D_{n,0}^{*}\backslash I^{*}\left(D_{n}^{*} \setminus S\left(x_1,\lambda_n,D_{n}^{*}\right )\right) \neq \emptyset$.
Similar
to the proof of Step 2, by the fact that $ D_{n,0}^{*}\backslash I^{*}\left(D_{n}^{*} \setminus S\left(x_1,\lambda_n,D_{n}^{*}\right )\right) \neq \emptyset$, there exist a point $y_{n+1}\in S(M)$ and a real number $\beta_{n+1}\in (0,\min \{\beta_{n}/256, \lambda_{n}/256$ $\})$  such that
\[
S\left(y_{n+1},D_{n,0}^{*},\beta_{n+1}\right)\subset D_{n,0}^{*}\backslash I^{*}\left(D_{n}^{*} \backslash S\left(x_1,\lambda_n,D_{n}^{*}\right )\right).~~~~~~\eqno~~~~~~~~~~~~~~(2.32)
\]
Therefore, from the above relation and $I^{*}\left(D_{n}^{*}\right)=D_{n,0}^{*}$, it is easy to see that
\[
S\left(y_{n+1},D_{n,0}^{*},\beta_{n+1}\right)\subset D_{n,0}^{*}\backslash I^{*}\left(D_{n}^{*} \backslash S\big(x_1,\lambda_n,D_{n}^{*}\right )\big)\subset                    I^{*}\left(S\big(x_1,\lambda_n,D_{n}^{*}\big )\right).
\]
We pick a point $y^{*}\in S\left(y_{n+1},D_{n,0}^{*},\beta_{n+1}\right)$. Since $\eta_{n}\in \left(0,\eta_{n-1}/4\right)$, by the formulas
(2.29) and (2.30),
there exists a point
$x^{*} \in S\left(x_1,\lambda_n,D_{n}^{*} \right)$ with $I^{*}\left(x^{*}\right)=y^{*}$ such that $\|x^{*}\|\leq 2$.
Noting that $I^{*}\left(x_0^{*}\right)=z^{*}$, $\left\|x-x_1\right\|<256^{-1}\beta_1\left(\gamma_0+1\right)^{-1}\left(\theta+1\right)^{-1}$ and $\lambda_n\in (0,\min\{\theta_{n},\overline{\eta}\} )$,
from $z^{*}(x)=a$ and $\left\|y_n-x_1\right\|<(3/4)\beta_1$, we obtain the following inequalities
\[
y^{*} (y_n)=x^{*} (y_n)~=~x^{*} \left(x_1\right)-x^{*} (x_1-y_n)\quad \quad\quad \quad\quad \quad\quad \quad
\]
\[
  \quad\quad\quad \quad \quad \quad\quad \quad\quad \quad\quad \quad\geq~ \sigma_{D_{n}^{*}}(x_1)~-~\lambda_n   ~-~x^{*} \left(x_1-y_n\right)\quad \quad \quad \quad\quad \quad \quad \quad\quad \quad \quad \quad\quad \quad \quad \quad\quad \quad \quad \quad\quad \quad \quad \quad\quad\quad \quad \quad \quad
\]
\[
 \quad \quad\quad \quad\quad \quad\quad \quad \quad \quad\quad \quad\geq~ x_0^{*}\left(x_1\right)~-~x^{*} \left(x_1-y_n\right)~-~\lambda_n\quad \quad \quad \quad\quad \quad \quad \quad\quad \quad \quad \quad\quad \quad \quad \quad\quad \quad \quad \quad\quad \quad \quad \quad\quad\quad \quad \quad \quad
\]
\[
  \quad\quad\quad \quad \quad\quad \quad\quad \quad \quad\quad \quad\geq~ x_0^{*}\left(x\right)~- ~ x_0^{*} \left(x-x_1\right)       ~-~x^{*} \left(x_1-y_n\right)~-~\lambda_n\quad \quad \quad \quad\quad \quad \quad \quad\quad \quad \quad \quad\quad \quad \quad \quad\quad \quad \quad \quad\quad \quad \quad \quad\quad\quad \quad \quad \quad
\]
\[
  \quad\quad\quad\quad \quad\quad \quad \quad \quad \quad\quad \quad\geq~ z^{*}\left(x\right)~-~  \left\|x_0^{*}\right\| \left\|x_1-x\right\|     ~ -~\left\|x^{*} \right\|\left\|x_1-y_n\right\|~-~\lambda_n\quad \quad \quad \quad\quad \quad \quad \quad\quad \quad \quad \quad\quad \quad \quad \quad\quad \quad \quad \quad\quad \quad \quad \quad
\]
\[
 \quad \quad\quad \quad \quad\quad \quad\quad \quad \quad\quad \quad\geq~ a~-~\frac{1}{256}\beta_1~-~\frac{6}{4}\beta_1~-~\lambda_n>\frac{1}{2}a. \quad \quad \quad \quad\quad \quad \quad \quad\quad \quad \quad \quad\quad \quad \quad \quad\quad \quad \quad \quad\quad \quad \quad \quad
\]
Since $y^{*}\in S\left(y_{n+1},D_{n,0}^{*},\beta_{n+1}\right)$ is arbitrary, from the above inequalities, we have
\[
S\left(y_{n+1},D_{n,0}^{*},\beta_{n+1}\right)\subset \left\{ x^{*}\in M^{*} :    x^{*}(y_n)>0                 \right\}.
\]
Noting that $I^{*}$ is a linear mapping and
$S\left(y_{n+1},D_{n,0}^{*},\beta_{n+1}\right)\subset I^{*}\left(S\left(x_1,\lambda_n,D_{n}^{*}\right )\right)$, by the formula (2.31),
we obtain the following formula
\[
S\left(y_{n+1},D_{n,0}^{*},\beta_{n+1}\right)- S\left({y_{n+1}},D_{n,0}^{*},\beta_{n+1}\right)\subset I^{*}(V_n),~~~\eqno~~~~~~~~~(2.33)
\]
where
\[
V_n= \left[ S\left(x_1,K_{1}^{*},\frac{1}{4}\eta_n \right)- S\left(x_1,K_{1}^{*},\frac{1}{4}\eta_n \right)+ B\left(0,\eta_n \right)\right].
\]
Noting that $D_{n,0}^{*}=\overline{co}^{w^{*}} (C^{*}\cup N_{n}^{*})$, $S\left(y_{n+1},D_{n,0}^{*},\beta_{n+1}\right)\subset \{ x^{*}\in M^{*} :    x^{*}(y_n)>0                 \}$ and $z^{*}(y_n)>0    $, by  the formula (2.33) and Lemma 2.4, we obtain that
\[
S\left(y_{n+1},C^*,\beta_{n+1}\right)-S\left(y_{n+1},C^*,\beta_{n+1}\right)\subset I^{*}(V_n),
\]
\[
S\left(y_{n+1},C^*,\beta_{n+1}\right)\subset \{ {x^*} \in {M^*}:{x^*}(y_n) > 0\}~~\quad~~~~~~\mathrm{and}~~~~~~~~~\quad~~~~~~~~~\big\|y_n-y_{n+1}\big\|<\frac{1}{2^{n}}\beta _n.
\]
Similar to  the proof of Step 1, we conclude that $S\left(y_{n+1},{C^*},\beta_{n+1}\right)\subset  S\left(y_n,{C^*},\beta_{n}\right)$.

\textbf{Step 4.} In this step, we complete the proof of  Theorem 2.1. Based on the iterative process from Step 1 to Step 3,
we construct two sequences $\{ {y_n}\} _{n = 1}^\infty  \subset S(M)$
and $\{ {\beta_n}\} _{n = 1}^\infty  \subset \mathbb R^{+}$ such that
\par (a) $\left\| {{y_{n + 1}} - {y_n}} \right\| < {\beta _n} \cdot {2^{ - n}}$ for every natural number $n\in \mathbb N$;
\par (b) $128{\beta _{n + 1}} < {\beta _n}$ for every natural number $n\in\mathbb N$;
\par (c) $S( {y_{n + 1}},{C^*},\beta _{n + 1} ) \subset S( {y_n},{C^*},{\beta _n} )\subset   I^{*}(V_n) $
for every natural number
$n\in\mathbb N$;
\par (d) $S\left( {y_{n + 1}},{C^*},\beta _{n + 1}\right ) \subset S\left( {y_n},{C^*},{\beta _n} \right)$
for every natural number $n\in \mathbb N$.
\\
Then, by conditions (a) and (b), we conclude that $\{ {y_n}\} _{n = 1}^\infty $ is a Cauchy sequence.
Since $X$ is complete, there exists a point $y\in X$ such that
$\| y_n -y\| \to 0$ as $n\to \infty$. Hence, for every $n\in N$, we obtain the following inequalities
\[
\big\| {y - {y_n}} \big\| ~=~ \mathop {\lim }\limits_{k \to \infty } \big\|   y_n -y_k  \big\|\quad\quad\quad\quad\quad\quad\quad\quad
\]
\[
\quad\quad\quad\quad\quad\quad \quad\quad\quad\quad\quad\quad=~ \mathop {\lim }\limits_{k \to \infty } \left\|\sum\nolimits_{i = 1}^{k - n}  \left[y_{n + i - 1} - y_{n + i} \right] \right\|\quad \quad\quad\quad\quad \quad\quad\quad\quad \quad\quad\quad\quad \quad\quad\quad\quad \quad\quad\quad\quad \quad\quad\quad
\]
\[
\quad\quad\quad\quad\quad\quad \quad\quad\quad\quad\quad\quad\le~ \mathop {\lim \sup }\limits_{k \to \infty } \left(\sum\nolimits_{i = 1}^{k - n} {\big\| y_{n + i - 1} - y_{n + i} \big\|}\right)\quad \quad\quad\quad\quad \quad\quad\quad\quad \quad\quad\quad\quad \quad\quad\quad\quad \quad\quad\quad\quad \quad\quad\quad
\]
\[
\quad\quad\quad\quad\quad\quad \quad\quad\quad\quad\quad\quad\le~ \mathop {\lim \sup }\limits_{k \to \infty } \left(\sum\nolimits_{i = 1}^{k - n} { \frac{1}{2^{ n + i-1}}\beta _n }\right)\quad\quad\quad\quad\quad\quad \quad\quad\quad\quad\quad\quad\quad\quad\quad\quad\quad\quad \quad\quad\quad\quad\quad\quad\quad\quad\quad\quad\quad\quad \quad\quad\quad\quad\quad\quad
\]
\[
\quad\quad\quad\quad\quad\quad \quad\quad\quad\quad\quad\quad\le~ \left(\frac{1}{2^{n-1}}\right) {\beta _n}.\quad\quad\quad \quad\quad\quad\quad \quad\quad\quad\quad \quad\quad\quad\quad \quad \quad\quad\quad\quad \quad\quad\quad\quad\quad\quad\quad\quad \quad\quad\quad\quad \quad\quad\quad
\]
We claim that $S( {y,{C^*},{\beta _n }/4} ) \subset S( y_n ,C^*,\beta _n  )$ whenever $n\geq 4$. In fact, we pick a point $y^{*}\in S\left( y,C^*,{\beta _n }/4 \right)$ whenever $n\geq 4$. Then
${y^*}(y) \ge \sup \left\{ {y^*}(y):y^*\in C^{*}\right\}  - ({\beta _n }/4)$ whenever $n\geq 4$.
Since the set $C^{*}$ is a weak$^{*}$ compact subset of $X^{*}$,  there exists
a point $y^{*}(n)\in S\left( {y_{n},{C^*},{\beta _n }/4} \right)$ such that
\[
\left\langle{y^*}( n ),y_{n}\right\rangle = \sup \left\{ {y^*}(y_{n}):{y^*} \in {C^*}    \right  \}.
\]
Therefore, since $C^* \subset B\left (M^*\right)$, we conclude that $\|y^{*}(n)\|\leq 1$ for all $n\in N$. Since  $y^{*}\in S\left( y,{C^*},{\beta _n }/4 \right)$ and $\left\| {y - {y_n}} \right\|\leq \beta _n/2^{n-1}$, we obtain the following inequalities
\[
{y^*}({y_n }) ~=~ {y^*}(y) - \big[{y^*}(y) - {y^*}({y_n })\big]~~~\quad\quad\quad \quad~~\quad\quad\quad\quad \quad~~~~~\quad\quad\quad \quad\quad\quad \quad~~~~~~~~~~~~~~~~~~~~~~~~~~~~~~~~~~~~~~~~~~~~~~~~~~~~~~~~~~~~~~~~~~~~~~~~~~~~~~~~~~~~~~~~~~~~~~~~~~~~~~~~~~~
\]
\[
\quad\quad\quad \quad\quad \quad \ge~ \sup \left\{ {y^*}(y):{y^*} \in {C^*}\right\}  ~-~ \frac{1}{4}{\beta _n} ~-~ \big [{y^*}(y) - {y^*}({y_n })\big]~~~~~~~~\quad \quad \quad\quad \quad \quad \quad \quad \quad \quad \quad \quad \quad \quad \quad \quad \quad \quad \quad\quad \quad \quad \quad \quad\quad\quad \quad \quad \quad \quad \quad ~~~~~~~~~~~~~~~~~~~~~~~~
\]
\[
\quad\quad\quad \quad\quad \quad  \ge~ \sup \left\{ {y^*}(y):{y^*} \in {C^*}\right\}  ~-~ \frac{1}{4}{\beta _n} ~-~ \left\| y^{*} \right\|\cdot\left\| y - {y_n } \right\|~~~~~~~~\quad \quad \quad\quad \quad \quad \quad \quad \quad \quad \quad \quad \quad \quad \quad \quad \quad \quad \quad\quad \quad \quad \quad \quad\quad\quad \quad \quad \quad \quad \quad ~~~~~~~~~~~~~~~~~~~~~~~~
\]
\[
\quad\quad\quad \quad\quad \quad  \ge~ \sup \left\{ {y^*}(y):{y^*} \in {C^*}\right\}  ~-~ \frac{1}{4}{\beta _n} ~-~ \left\| y - {y_n } \right\|~~~~~~~~\quad \quad \quad\quad \quad \quad \quad \quad \quad \quad \quad \quad \quad \quad \quad \quad \quad \quad \quad\quad \quad \quad \quad \quad\quad\quad \quad \quad \quad \quad \quad ~~~~~~~~~~~~~~~~~~~~~~~~
\]
\[
\quad\quad\quad \quad\quad \quad  \ge~\left\langle{y^*}(n ),y\right\rangle~-~ \frac{1}{4}{\beta _n} ~-~ \left\| y - y_n  \right\|~~~~~~~~\quad \quad \quad\quad \quad \quad \quad \quad \quad \quad \quad \quad \quad \quad \quad \quad \quad \quad \quad\quad \quad \quad \quad \quad\quad\quad \quad \quad \quad \quad \quad ~~~~~~~~~~~~~~~~~~~~~~~~
\]
\[
\quad\quad\quad \quad\quad \quad\ge~ \langle{y^*}(n ),{y_n }\rangle ~-~ \big[\langle{y^*}(n),{y_n }\rangle -\langle{y^*}(n ),y\rangle    \big] ~-~ \frac{1}{4}{\beta _n } ~-~\left\| {y - {y_n }} \right\|~~~~\quad \quad \quad\quad \quad \quad \quad \quad \quad \quad \quad \quad \quad \quad \quad \quad \quad \quad \quad\quad \quad \quad \quad \quad\quad\quad \quad \quad \quad \quad \quad ~~~~~~~~~~~~~~~
\]
\[
\quad\quad\quad \quad\quad \quad\ge~ \langle{y^*}(n ),{y_n}\rangle ~-~ \left\|{y^*}(n)\right\|\cdot\left\| y - {y_n } \right\| ~-~ \frac{1}{4}{\beta _n } ~-~ \left\| y - {y_n } \right\|~~\quad \quad \quad\quad \quad \quad \quad \quad \quad \quad \quad \quad \quad \quad \quad \quad \quad \quad \quad\quad \quad \quad \quad \quad\quad\quad \quad \quad \quad \quad \quad ~~~~~~~
\]
\[
\quad\quad\quad \quad\quad \quad\ge~ \langle{y^*}(n ),{y_n }\rangle ~-~ \frac{1}{8}{\beta _n} ~-~ \frac{1}{4}{\beta _n} ~-~ \frac{1}{8}{\beta _n }~~~\quad \quad \quad\quad \quad \quad \quad \quad \quad \quad \quad \quad \quad \quad \quad \quad \quad \quad \quad\quad \quad \quad \quad \quad\quad\quad \quad \quad \quad \quad \quad ~~
\]
\[
\quad\quad\quad \quad\quad \quad\ge~ \sup \left\{ {y^*}(y_{n}):{y^*} \in {C^*}\right\}   ~-~\frac{1}{2}{\beta _n }~\quad \quad \quad\quad \quad \quad \quad \quad \quad \quad \quad \quad \quad \quad \quad \quad \quad \quad \quad\quad \quad \quad \quad \quad\quad\quad \quad \quad \quad \quad \quad ~~~
\]
for all $n\geq 4$.
This implies that
$
{y^*}({y_n})\ge \sup \left\{ {y^*}(y_{n}):{y^*} \in {C^*}\right\}  - \beta _n
$
for all $n\geq 4$. Therefore, since ${y^*}  \in {C^*}$, we obtain that ${y^*}\in S\left( {y_{n},{C^*},{\beta _n }} \right)$ for all $n\geq 4$. Since $y^{*}\in S\left( y,C^*,\beta _n /4 \right)$ is arbitrary,
we conclude that
\[
S\left( {y,{C^*},\frac{1}{4}{\beta _n }} \right) \subset S\left( {{y_n },{C^*},{\beta _ n}} \right)~~\quad~~~~~~\mathrm{whenever}~~~~~~~~~\quad~~~~~~~~~ n\geq 4.~~~~\eqno~~~~~~~~~(2.34)
\]
Let $S\left( y,C^* ,0\right)= \{y^{*}\in C^*: y^{*}(y)= \sigma_{C^*}(y) \}$. Then $S\left( y,C^* ,0\right)\neq \emptyset$.
Since
$C^{*}$ is a weak$^{*}$  bounded closed convex subset of $X^{*}$, we conclude that $S\left( y,C^* ,0\right)$ is a
nonempty weak$^{*}$ bounded closed convex subset of $X^{*}$.
Therefore, by the
formula (2.34)
and
condition (c), we have the following  inclusion relations
\begin{eqnarray*}
\quad \quad\quad S\left( {y,{C^*},0} \right)-S\left( {y,{C^*},0} \right)
&\subset& S\left( {y,{C^*},\frac{1}{4}{\beta_{n+1}}} \right)-S\left( y,{C^*},\frac{1}{4}{\beta_{n+1}} \right)
\\
&\subset& S\left( {y_{n + 1}},{C^*},{\beta_{n+1}} \right) - S\left( {y_{n + 1}},{C^*},{\beta_{n+1}} \right)
\\
&\subset& I^{*}\left(V_n\right)  \quad  \quad \quad  \quad \quad  \quad \quad  \quad\quad  \quad \quad \quad \quad  \quad ~(2.35)
\end{eqnarray*}
for all natural number $n\geq 4$. Moreover, it is easy to see that if the nonempty set $S\left( y,C^* ,0\right)$ is a singleton, then the point is a weak$^{*}$
exposed point of $C^*$. Suppose that $S\left( y,C^* ,0\right)$ is not a singleton. Then there exist
two points $y_{1}^{*}\in  S\left( y,C^*,0 \right)$ and $y_{2}^{*} \in  S\left( y,C^* ,0\right)$ such that
$y_{1}^{*} \neq y_{2}^{*}$. Therefore, since $y_{1}^{*} \in C^*$ $ \subset M^{*}$ and $y_{2}^{*} \in $ $C^* \subset M^{*}$,
there exists a point $x_0 \in S(E)$ such that
\[
\big\langle  y_{1}^{*} -y_{2}^{*} ,    x_0      \big\rangle \geq \frac{1}{2} \left\| y_{1}^{*} -y_{2}^{*}\right\|>0.~~~~~~~~~~~~~~~~~~~~~~~~~~~~~~~~~~~~~~~~\eqno~~~~~~~~~~~~~~~~~~~~~~~~~~~~~~~~~~~~~(2.36)
\]
Hence we define the weak$^{*}$ neighborhood $V_0$  of the origin in $M^{*}$, where
\[
V_0= \left\{ x^{*} \in M^{*}:    \left| \left\langle x^{*} ,    x_0      \right\rangle \right|<   \frac{1}{10} \left\| y_{1}^{*} -y_{2}^{*}\right\|         \right\}.
\]
Since $\sigma_{K_{1}^{*}}$ is G$\mathrm{\mathrm{\hat{a}}}$teaux differentiable $x_1\in X$ and $\eta_n \to0$, by Lemma 2.2, we may assume without loss of generality that
\[
S\left(x_1,K_{1}^{*},\frac{1}{4}\eta_n \right)- S\left(x_1,K_{1}^{*},\frac{1}{4}\eta_n \right)\subset \left\{ x^{*} \in X^{*}:    \left| \left\langle x^{*} ,    x_0      \right\rangle \right|<   \frac{1}{64} \left\| y_{1}^{*} -y_{2}^{*}\right\|         \right\}.
\]
for every natural number $n\in \mathbb N$. Moreover, since $\eta_n \to0$, we may assume without loss of generality that $\eta_n<128^{-1}\left\| y_{1}^{*} -y_{2}^{*}\right\| $ for every $n\in\mathbb N$. It follows that
\[
B\left(0,\eta_n \right)\subset \left\{ x^{*} \in X^{*}:    \left| \left\langle x^{*} ,    x_0      \right\rangle \right|<   \frac{1}{64} \left\| y_{1}^{*} -y_{2}^{*}\right\|         \right\}
\]
for every natural number $n\in\mathbb N$. Therefore, from the previous proof, we have
\begin{eqnarray*}
 V_n &=& \left[ S\left(x_1,K_{1}^{*},\frac{1}{4}\eta_n \right)- S\left(x_1,K_{1}^{*},\frac{1}{4}\eta_n \right) +B\left(0,\eta_n \right)\right]
\\
& \subset& \left\{ x^{*} \in X^{*}:    \left| \left\langle x^{*} ,    x_0      \right\rangle \right|<   \frac{1}{64} \left\| y_{1}^{*} -y_{2}^{*}\right\|         \right\}
\\
&& + \left\{ x^{*} \in X^{*}:    \left| \left\langle x^{*} ,    x_0      \right\rangle \right|<   \frac{1}{64} \left\| y_{1}^{*} -y_{2}^{*}\right\|         \right\}
\\
&\subset&  \left\{ x^{*} \in X^{*}:    \left| \left\langle x^{*} ,    x_0      \right\rangle \right|<   \frac{1}{32} \left\| y_{1}^{*} -y_{2}^{*}\right\|         \right\}.
\end{eqnarray*}
Therefore, by the definition of $I^{*}$  and $x_0 \in S\left(E\right)$, we obtain the following formula
\begin{eqnarray*}
&&I^{*}\left\{ x^{*} \in X^{*}:    \left| \left\langle x^{*} ,    x_0      \right\rangle \right|<   \frac{1 }{32} \left\| y_{1}^{*} -y_{2}^{*}\right\|         \right\}
\\
&=&\left\{ x^{*} \in M^{*}:    \left| \left\langle x^{*} ,    x_0      \right\rangle \right|<   \frac{1}{32} \left\| y_{1}^{*} -y_{2}^{*}\right\|         \right\}.
\end{eqnarray*}
Since $x_0 \in S\left(E\right)$, from the above relation and the definition of \(V_0\), we obtain that
\begin{eqnarray*}
I^{*}\left(V_n\right)&\subset& I^{*}\left(\left\{ x^{*} \in X^{*}:    \left| \left\langle x^{*} ,    x_0      \right\rangle \right|<   \frac{1}{32} \left\| y_{1}^{*} -y_{2}^{*}\right\|         \right\}\right)
\\
&=&\left\{ x^{*} \in M^{*}:    \left| \left\langle x^{*} ,    x_0      \right\rangle \right|<   \frac{1}{32} \left\| y_{1}^{*} -y_{2}^{*}\right\|         \right\}
\\
&\subset& \left\{ x^{*} \in M^{*}:    \left| \left\langle x^{*} ,    x_0      \right\rangle \right|<   \frac{1}{10} \left\| y_{1}^{*} -y_{2}^{*}\right\|         \right\} = V_0
\end{eqnarray*}
for every natural number $n\in\mathbb N$. Therefore, by the formula (2.35) and the above formula, we conclude that
\begin{eqnarray*}
y_{1}^{*} - y_{2}^{*}&\in& S\left( {y,{C^*}},0 \right)-S\left( {y,{C^*}},0 \right)
\\
&\subset& S\left( {y_{n + 1} },{C^*},{\beta_{n+1}} \right) - S\left( {y_{n + 1} },{C^*},{\beta_{n+1}} \right)
\\
&\subset& I^{*}\left(V_n\right)
\subset V_0.
\end{eqnarray*}
Then, by the definition of $ V_0$, we have $10^{-1}\left\| y_{1}^{*} -y_{2}^{*}\right\| > \left|  \left\langle  y_{1}^{*} -y_{2}^{*} ,    x_0      \right\rangle  \right|$.
Therefore, by the inequality $10^{-1}\left\| y_{1}^{*} -y_{2}^{*}\right\| > \left|  \left\langle  y_{1}^{*} -y_{2}^{*} ,    x_0      \right\rangle  \right|$
and the formula (2.35), we have
\[
\frac{1}{10} \left\| y_{1}^{*} -y_{2}^{*}\right\| > \left|  \left\langle  y_{1}^{*} -y_{2}^{*} ,    x_0      \right\rangle  \right| \geq  \left\langle  y_{1}^{*} -y_{2}^{*} ,    x_0      \right\rangle \geq \frac{1}{2} \left\| y_{1}^{*} -y_{2}^{*}\right\|>0,
\]
a contradiction. Hence we obtain that
$S\left( y,{C^*},0 \right)$ is a singleton. This implies that
the functional $\sigma_{C^*}$ is G$\mathrm{\hat{a}}$teaux differentiable at $y\in M$. Since $\|y_n-y_{n+1}\big\|<\beta _n/{2^{n}}$ and
$\beta_{1}=\min\left\{ r/256,a/256\right\}$, we conclude that $y\in B\left(x,r\right)\cap M$.
We claim that $C^*=\overline {co}^{w^{*}}  \left( {w^*}\exp {C^*} \right)$ holds. In fact,
suppose that
\[
{C^*}\backslash \overline {co}^{w^{*}}  \left( {{w^*}\exp {C^*}} \right) \ne  \emptyset.
\]
Then there exists a point $y_{0}^{*}\in {C^*}$ such that $y_{0}^{*}\in  {C^*}\backslash \overline {co}^{w^{*}}  \left( {{w^*}\exp {C^*}} \right)$.
Hence
\[
y_{0}^{*}\notin \overline {co}^{w^{*}}  \left( {{w^*}\exp {C^*}} \right) \ne  \emptyset. ~~~~~~~~~~~~~~~~~~~~~~~~~~~~~~~~~~~~\eqno~~~~~~~~~~~~~~~~~~~~~~~~~~~~~(2.37)
\]
Therefore, by the separation theorem and the formula (2.37),
there exists a point $z\in S\left(M\right)$ and a real number $\eta\in(0,1)$ such that
\[
 y_0^*\left(z\right) - 6\eta  \ge \sup \left\{ {z^*}(z):{z^*} \in \overline {co}^{w^{*}}  \left( {{w^*}\exp {C^*}} \right)\right\}.~~~~~~~~~~~~~~~~~~~~~~~~~~~~\eqno~~~~~~~~~~~~~~~~~~~~~~~~~~~~~(2.38)
\]
Pick a point $y^*\in S\left(z,C^*,\eta \right)$. Then we get that $y^*\left(z\right)\geq \sigma_{C^*}(z)-\eta >y_0^*\left(z\right) - 6\eta$.
Since $y^*\in S\left(z,C^*,\eta \right)$ is arbitrary, by the formula (2.38), we conclude that
\[
S\left(z,C^*,\eta \right) \cap \overline {co}^{w^{*}}  \left( {{w^*}\exp {C^*}} \right) = \emptyset. ~~~~~~~~~~~~~~~~~~~~~~~~~~~~~~~~~~~~~~~~\eqno~~~~~~~~~~~~~~~~~~~~~~~~~~~~~(2.39)
\]
Since the functional $\sigma_{C^*}$ is G$\mathrm{\hat{a}}$teaux differentiable on a dense subset of $M$, there exists a point $z_1\in S(M)$ with $\|z-z_1\|<\eta /16$ such that
$\sigma_{C^*}$ is G$\mathrm{\hat{a}}$teaux differentiable at $z_1\in S(M)$. Then
$ d\sigma_{C^*}(z_1)$ is a weak$^{*}$ exposed point of $C^*$ and is exposed by $z_1$,
we obtain that $\langle d\sigma_{C^*}(z_1), z_1\rangle =\sigma_{C^*}(z_1)$. Since $\langle d\sigma_{C^*}(z_1), z_1\rangle =\sigma_{C^*}(z_1)$ and $C^*\subset B\left(M^{*}\right)$, we conclude that
\begin{eqnarray*}
\big\langle d\sigma_{C^*}(z_1), z \big\rangle &=& \big\langle d\sigma_{C^*}(z_1), z_1 \big\rangle- \big\langle d\sigma_{C^*}(z_1), z_1- z \big\rangle
\\
&=&\sigma_{C^*}(z_1)- \big\langle d\sigma_{C^*}(z_1), z_1- z \big\rangle
\\
&\geq& \sigma_{C^*}(z)- \sigma_{C^*}(z-z_1) - \big\langle d\sigma_{C^*}(z_1), z_1- z \big\rangle
\\
&\geq& \sigma_{C^*}(z)- \big\|z-z_1\big\| - \big\|z-z_1\big\|
\\
&\geq & \sigma_{C^*}(z)- \frac{1}{8}\eta.
\end{eqnarray*}
Hence we obtain  that $ d\sigma_{C^*}(z_1)\in S\left(z,C^*,\eta /8\right)\subset S\left(z,C^*,\eta \right) $. Since $ d\sigma_{C^*}(z_1)$ is a weak$^{*}$ exposed point of $C^*$ and is exposed by $z_1$, we conclude that
\[
S\left(z,C^*,\eta \right) \cap \overline {co}^{w^{*}}  \left( {{w^*}\exp {C^*}} \right) \neq \emptyset,
\]
a contradiction. Hence we get that ${C^*}=\overline {co}^{w^{*}}  ( {{w^*}\exp {C^*}} )$.
Then $\overline{co}^{w^{*}}(w^{*}\mathrm{exp} C^{*})$ has the non-empty interior.
Pick a point $z_0\in S\left(M\right)$. Then, by the Hahn-Banach theorem, there exists a
functional $z_0^{*}\in S\left(M^{*}\right)$ such that $z_0^{*}(z_0)$ $=\|z_0\|=1$.
Hence we define the closed subspace $N(z_0^{*})$ of $M^{*}$, where
\[
N\left(z_0^{*}\right)= \{ x\in M  : z_0^{*}(z)=0           \}.
\]
Since $\overline{co}^{w^{*}}(w^{*}\mathrm{exp} C^{*})$ has the non-empty interior for every
symmetric inner non-empty weak$^{*}$ closed bounded convex set $C^{*}\subset M^{*}$,
by Lemma 2.5, we obtain that $N(z_0^{*})$ is a G$\mathrm{\hat{a}}$teaux differentiability space.
Moreover, since
$N(z_0^{*})$ is a hyperplane of $M$, we obtain that
$M$ is a G$\mathrm{\hat{a}}$teaux differentiability space, which completes the proof.
\end{proof}

\begin{corollary}
Suppose that $M$ is a closed subspace of a weak Asplund space $X$. Then $M$ is a  G$\mathrm{\hat{a}}$teaux differentiability space.

\end{corollary}

We now present some new criteria for G$\mathrm{\hat{a}}$teaux differentiability spaces based on the preceding lemmas.

\begin{theorem}
Suppose that $X$ is a Banach space. Then the following statements are equivalent:

\par (1) $X$ is a G$\hat{a}$teaux differentiability space;

\par (2) $\overline{co}^{w^{*}}(w^{*}\mathrm{exp} C^{*})$ has the non-empty interior for every
symmetric inner non-empty weak$^{*}$ closed bounded convex set $C^{*}\subset X^{*}$;

\par (3) $\overline{co}^{w^{*}}(w^{*}\mathrm{exp} C^{*})$ has the non-empty interior for every
inner non-empty weak$^{*}$ closed bounded convex set $C^{*}\subset X^{*}$;

\par (4) For any
symmetric inner non-empty weak$^{*}$ closed bounded convex set $C^{*}\subset$ $ X^{*}$, the functional $\sigma_{C^{*}}$ is G${\hat{a}}$teaux differentiable on a dense subset of $X$;

\par (5) For any
inner non-empty weak$^{*}$ closed bounded convex set $C^{*}\subset X^{*}$,  the functional $\sigma_{C^{*}}$ is G${\hat{a}}$teaux differentiable on a dense subset of $X$.

\end{theorem}

\begin{proof}
It is well known that $X$ is a G$\mathrm{\hat{a}}$teaux differentiability space if and only if every bounded weak$^{*}$
closed convex subset of $X^{*}$ is the weak* closed convex hull of its weak$^{*}$ exposed points.

\par Pick an inner non-empty weak$^{*}$ closed bounded convex subset $C^{*}$ of $X^{*}$ so that
$\sigma_{C^{*}}$ is G$\mathrm{\hat{a}}$teaux differentiable on a dense subset of $X$. It follows that
\[
{C^*}=\overline {co}^{w^{*}}  \big( {{w^*}\exp {C^*}} \big).
\]
Since $C^{*}$ is an inner non-empty weak$^{*}$ closed bounded convex subset of $X^{*}$, from ${C^*}=\overline {co}^{w^{*}}  ( {{w^*}\exp {C^*}} )$, we obtain that $\overline{co}^{w^{*}}(w^{*}\mathrm{exp} C^{*})$ has the non-empty interior.
Then, by Lemma 2.5, it is easy to see that Theorem 2.7 is true, which completes the proof.
\end{proof}

\textbf{Acknowledgement.} This research is supported by "China Natural Science Funds under grant 12271121 and 11561053".

\textbf{Declaration.} This manuscript represents the author's original work, and no AI tools were used for text generation, revision or assistance throughout the research and writing process.

\bibliographystyle{amsplain}

\end{document}